\documentclass[a4paper]{article}

\usepackage{amssymb}
\usepackage[pdftex,bookmarksnumbered=true]{hyperref}
\usepackage{pgf,tikz}
\usepackage[center]{caption}  % Pour les \caption multilignes.
\newtheorem{lemma}[subsection]{Lemma}
\newtheorem{proposition}[subsection]{Proposition}
\newtheorem{theorem}[subsection]{Theorem}
\newtheorem{corollary}[subsection]{Corollary}
\newtheorem{fact}[subsection]{Fact}
   {\refstepcounter{subsection}%
        \medbreak\noindent{\bf Question \thesubsection\space}}%
   {\par\medbreak}%
\newenvironment{remark}%
   {\refstepcounter{subsection}%
        \medbreak\noindent{\bf Remark \thesubsection\space}}%
   {\par\medbreak}%
\newenvironment{example}%
   {\refstepcounter{subsection}%
        \medbreak\noindent{\bf Example \thesubsection\space}}%
   {\par\medbreak}%
\newenvironment{proof}%
   {\medbreak\noindent{\it Proof:\space}}%
   {\hfill$\dashv$\smallbreak}%
\newcommand{\df}{\bf}
\let\sauvegardetiret=\-
\renewcommand{\-}[1]{\ifx#1-\penalty10000\hbox{-\relax}\penalty10000\else\sauvegardetiret#1\fi}
\long\def\footnotesymbol[#1]#2{\begingroup%
\def\thefootnote{\fnsymbol{footnote}}\footnote[#1]{#2}\endgroup}
\newcommand{\tq}{\,\big/\ }
\newcommand{\llor}{\mathop{\lor\mskip-6mu\relax\lor}}
\newcommand{\vvee}{\llor}
\newcommand{\lland}{\mathop{\land\mskip-6mu\relax\land}}
\newcommand{\wwedge}{\lland}
\newcommand{\mmeet}{\wwedge}
\newcommand{\jjoin}{\vvee}
\newcommand{\meet}{\wedge}
\newcommand{\join}{\vee}

\let\olddownarrow=\downarrow
\renewcommand{\downarrow}{{\olddownarrow}}
\let\olduparrow=\uparrow
\renewcommand{\uparrow}{{\olduparrow}}

\renewcommand{\tq}{\mathop{:}}

\newcommand{\conj}{\bigwedge}
\newcommand{\cconj}{\mathop{\conj\mskip-9mu\relax\conj}}
\newcommand{\disj}{\bigvee}
\newcommand{\ddisj}{\mathop{\bigvee\mskip-9mu\relax\bigvee}}

\newcommand{\llat}{{\cal L}_{\rm lat}}
\newcommand{\ltc}{{\cal L}_{\rm HA^*}}
\newcommand{\lha}{{\cal L}_{\rm HA}}

\newcommand{\UN}{{\rm\bf 1}}
\newcommand{\ZERO}{{\rm\bf 0}}
\newcommand{\cC}{{\cal C}}

\newcommand{\cH}{{\cal H}}
\newcommand{\cI}{{\cal I}}

\newcommand{\cV}{{\cal V}}

\newcommand{\Lneuf}{{\rm\bf L}_9}
\newcommand{\Lcinq}{{\rm\bf L}_5}
\newcommand{\Ltrois}{{\rm\bf L}_3}
\newcommand{\Ldeux}{{\rm\bf L}_2}
\newcommand{\Lun}{{\rm\bf L}_1}

\newcommand{\dec}{\mathop{\mathcal{L}^\downarrow}}

\newcommand{\jirr}{\mathop{\mathcal{I}^\join}}
\newcommand{\gen}[1]{\langle #1\rangle}
\newcommand{\card}{\mathop{\rm Card}}

\newcommand{\dessin}{%

  \begin{scope}[every node/.style={draw,shape=circle,fill=black, inner sep=0mm, minimum size=1mm}]

    \draw (-1,1) -- (-1,2) node{}
        (-2,2) -- (-2,3) node{}
        (0,2) node{} -- (0,3)
        (-1,3) node{} -- (-1,4);

    \draw (0,0) -- (-2,2)
      (1,1) node{} -- (-1,3)
      (-1,2) -- (-2,3)
      (0,3) -- (-1,4);

  \end{scope}

  \begin{scope}[every node/.style={draw,shape=circle,fill=white, inner sep=0mm, minimum size=1mm}]

    \draw (0,0) node{} -- (1,1) 
      (-1,1) node{} -- (0,2)
      (-2,2) node{} -- (-1,3)
      (-1,2) -- (0,3) node{}
      (-2,3) -- (-1,4) node{};
  \end{scope}
}

\title{Model completion of varieties of co-Heyting algebras}
\author{
\begin{minipage}{6cm} 
Luck Darni\`ere \\ 
\begin{small}
D\'epartement de math\'ematiques \\[-1ex] 
Facult\'e des sciences \\[-1ex] 
Universit\'e d'Angers, France
\end{small}
\end{minipage} 
\begin{minipage}{6cm} 
Markus Junker \\
\begin{small}
Mathematisches Institut \\[-1ex] 
Abteilung f\"ur mathematische Logik \\[-1ex] 
Universit\"at Freiburg, Deutschland
\end{small}
\end{minipage}
\bigskip
}

\begin{document}

\maketitle

\begin{abstract}\scriptsize
It is known that exactly eight varieties of Heyting algebras have a
model-completion. However no concrete axiomatization of these
model-completions were known by now except for the trivial variety
(reduced to the one-point algebra) and the variety of Boolean
algebras. For each of the six remaining varieties we introduce two
axioms and show that 1) these axioms are satisfied by all the algebras
in the model-completion, and 2) all the algebras in this variety
satisfying these two axioms satisfy a certain remarkable embedding
theorem. For four of these six varieties (those which are locally
finite) these two results provide a new proof of the existence of a
model-completion with, in addition, an explicit and finite
axiomatization.
\end{abstract}

\section{Introduction}%
\label{se:intro}

It is known from a result of Maksimova \cite{maks-1977} that there are
exactly eight varieties of Heyting algebras that have the amalgamation
property (numbered $\cH_1$ to $\cH_8$, see section~\ref{se:prereq}).
Only these varieties (more exactly their theories) can have a model
completion\footnote{Basic model theoretic notions are recalled in
section~\ref{se:prereq} but we may already point out a remarkable
application of the existence of a model-completion for these
varieties, namely that for each of the corresponding
super-intuitionistic logics the second order propositional calculus
IPC$^2$ is interpretable in the first order propositional calculus
IPC$^1$ (in the sense of \cite{pitt-1992}).} and it is known since the
1990's that this is indeed the case\footnote{See
\cite{ghil-zawa-1997}. For $\cH_3$ to $\cH_8$, which are locally
finite, the existence of a model-completion follows from the
amalgamation property and \cite{whee-1976}, corollary~5. For the
variety $H_1$ of all Heyting algebras it is a translation in
model-theoretic terms of a theorem of Pitts \cite{pitt-1992} combined
with \cite{maks-1977}, as is explained in \cite{ghil-zawa-1997}. It is
also claimed in \cite{ghil-zawa-1997} that the same holds for $\cH_2$
up to minor adaptations of \cite{pitt-1992}.}. On the other hand no
model-theoretic proof of these facts were known until now and these
model-completions still remain very mysterious except for $\cH_7$ and
$\cH_8$: the latter is the trivial variety reduced to the one point
Heyting algebra, and the former is the variety of Boolean algebras
whose model completion is well known. 

In this paper we partly fill this lacuna by giving new proofs for some
of these results using algebraic and model-theoretic methods, guided
by some geometric intuition. We first give in
section~\ref{se:minim-exten} a complete classification of all the
minimal finite extensions of a Heyting algebra $L$. This is done by
proving that these extension are in one-to-one correspondence with
certain special triples of elements of $L$. Each of the remaining
sections~\ref{se:V1} to \ref{se:V6} is devoted to one of the varieties
$\cH_i$. We introduce for each of them two axioms that we call
``density'' and ``splitting'' and prove our main results:

\begin{theorem}\label{th:EC-intro}
  Every existentially closed model of $\cH_i$ satisfies the density
  and splitting axioms of $\cH_i$.
\end{theorem}

\begin{theorem}\label{th:embed-intro}
  Given a Heyting algebra $L$ in $\cH_i$ and a finite substructure $L_0$
  of $L$, if $L$ satisfies the density and splitting axioms of $\cH_i$
  then every finite extension $L_1$ of $L_0$ admits an embedding into
  $L$ which fixes $L_0$ pointwise. 
\end{theorem}

By standard model-theoretic arguments (see fact~\ref{fa:mod-comp}) it
follows immediately that if $\cH_i$ is locally finite then $\cH_i$ has
a model-completion which is axiomatized by the axioms of Heyting
algebras augmented by the density and the splitting axioms of $\cH_i$.
So this gives a new proof of the previously known model-completion
results for $\cH_3$ to $\cH_6$, which provides in addition a simple
axiomatization of these model completions.

Unfortunately we do not fully recover the existence of a
model-completion for $\cH_1$ and $\cH_2$. However our axioms shed some
new light on the algebraic structure of the existentially closed
Heyting algebras in these varieties. Indeed it is noticed in
\cite{ghil-zawa-1997} that such algebras satisfy the density axiom of
$\cH_1$, but neither the splitting property nor any condition
sufficient for theorem~\ref{th:embed-intro} to hold, seem to have been
suspected until now. Moreover all the algebraic properties of
existentially closed Heyting algebras in $\cH_1$ which are listed in
\cite{ghil-zawa-1997} can be derived from our two axioms, as we shall
see in the appendix. 

Let us also point out an easy consequences of
theorem~\ref{th:embed-intro}.

\begin{corollary}
  If $L$ is an algebra in $\cH_i$ which satisfies the density and
  splitting axioms of $\cH_i$ then every finite algebra in $\cH_i$
  embeds into $L$, and every algebra in $\cH_i$ embeds into an
  elementary extension of $L$. 
\end{corollary}

\begin{remark}
  In this paper we do not actually deal with Heyting algebras but with
  their duals, obtained by reversing the order. They are often called
  {\df co\--Heyting algebras} in the literature. Readers familiar
  with Heyting algebras will certainly find annoying this
  reversing of the order. We apologise for this discomfort but there
  are good reasons for doing so. Indeed, the present work has been
  entirely build on a geometric intuition coming from the
  fundamental example\footnote{This geometric intuition also played a
  role in the very beginning of the study of Heyting algebras. Indeed,
  co-Heyting algebras were born Brouwerian lattices in the paper of
  Mckinsey and Tarski \cite{mcki-tars-1946} which originated much of
  the later interest in Heyting algebras. Also the introduction of
  ``slices'' in \cite{hoso-1967} seems inspired by the same geometric
  intuition that we use in this paper. Indeed, the dual of the
  co-Heyting algebra of all subvarieties of an algebraic variety $V$
  belongs to the $d+1$\--th slice if and only if $V$ has dimension $\leq
  d$ (see \cite{darn-junk-2011} for more on this topic).}  of the
  lattice of all subvarieties of an algebraic variety, and their
  counterparts in real algebraic geometry. Such lattices
  are co-Heyting algebras, not Heyting algebras. To see how this
  intuition is used in finding the proofs (and then hidden while writing
  the proofs) look at figure~\ref{fi:glueing} in lemma~\ref{le:S1}.
\end{remark}

\section{Other notation, definitions and prerequisites}%
\label{se:prereq}

We denote by $\llat=\{\ZERO,\UN,\join,\meet\}$ the language of distributive
bounded lattice, these four symbols referring respectively to the
least element, the greatest element, the join and meet operations.
$\lha=\llat\cup\{\to\}$ and $\ltc=\llat\cup\{-\}$ are the language of Heyting
algebras and co-Heyting algebras respectively. Finite joins and meets
will be denoted $\jjoin$ and $\mmeet$, with the natural convention that the
join (resp. meet) of an empty family of elements is $\ZERO$ (resp.
$\UN$). 

The logical connectives `and', `or', and their iterated forms
will be denoted $\conj$, $\disj$, $\cconj$ and $\ddisj$ respectively.

We denote by $\jirr(L)$ the set of {\df join irreducible} elements of
a lattice $L$, that is the elements $a$ of $L$ which can not be written as the
join of any finite subset of $L$ not containing $a$. Of course $\ZERO$
is never join irreducible since it is the join of the empty subset of
$L$. The set $\jirr(L)$ inherits the order induced by $L$.

\paragraph{Dualizing rules.}
In order to help the reader more familiar with Heyting algebras than
co-Heyting algebras, we recommend
the use of the following conversion rules. For any ordered set $L$ we
denote by $L^*$ the {\df dual} of $L$, that is the same set with the
reverse order. If $a$ is an element of $L$ we denote by $a^*$ the same
element {\em seen as en element of $L^*$}, so that we can write for
any $a,b \in L$:
\begin{displaymath}
  a \leq b \iff b^* \leq a^* 
\end{displaymath}
Indeed the star indicates that the second symbol $\leq$ refers to the
order of $L^*$, and the first one to the order of $L$. Similarly if $L$ is a
co-Heyting algebra we can write:
\begin{eqnarray*}
  \ZERO^* = \UN         &\hbox{and}& \UN^* =\ZERO\\
  (a \join b)^* = a^* \meet b^* &\hbox{and}& (a \meet b)^* = a^* \join b^*
\end{eqnarray*}
The minus operation of $\ltc$ stands of course for the dual of the arrow
operation of $\lha$, but beware of the order of the operands:
\begin{displaymath}
  a-b = \min\{c\tq a \leq b \join c\} = (b^*\to a^*)^* 
\end{displaymath}
The topological symmetric difference is defined as:
\begin{displaymath}
  a \bigtriangleup b = (a-b) \join (b-a) = (a^* \leftrightarrow b^*)^*
\end{displaymath}
We will make extensive use of the following relation:
\begin{displaymath}
  b \ll a \iff a-b = a \mbox{ and }b \leq a 
\end{displaymath}
Note that $b \ll a$ and $b\not<a$ if and only if $a=b=\ZERO$, hence $\ll$ is a
strict order on $L \setminus\{\ZERO\}$. Note also that if $a$ is join
irreducible in $L$ (hence non zero) then $b \ll a$ if and only if $b<a$.

\paragraph{The varieties of Maksimova.}
We can now describe the varieties $\cH_1$ to $\cH_8$ introduced by
Maksimova, more exactly the corresponding varieties $\cV_1$ to $\cV_8$ of
co-Heyting algebras. Note that the intuitionistic negation $\lnot \varphi$ being
defined as $\varphi \to \bot$, the corresponding operation in co-Heyting algebras
is $\UN-a$:
\begin{displaymath}
  \UN-a=(a^* \to \UN^*)^*=(a^* \to \ZERO)^*=\big(\lnot (a^*)\big)^*
\end{displaymath}

\begin{itemize}
  \item 
    $\cV_1$ is the variety of all co-Heyting algebras.
  \item
    $\cV_2=\cV_1+\big[(\UN-x) \meet (\UN-(\UN-x))=\ZERO\big]$
    is the dual of the variety of the logic of the weak excluded middle
    ($\lnot x \join  \lnot \lnot x = \UN$).
  \item
    $\cV_3=\cV_1+\big[\big(((\UN-x)\meet x)-y\big)\meet y=\ZERO\big]$
    is the dual of the second slice of Hosoi. With the
    terminology of \cite{darn-junk-2011}, $\cV_3$ is the variety of
    co-Heyting algebras of dimension $\leq 1$. So a co-Heyting algebra
    $L$ belongs to $\cV_3$ if and only if any prime filter of $L$
    which is not maximal is minimal (with respect to inclusion, among
    the prime filters of $L$).  
  \item
    $\cV_4=\cV_3+\big[(x-y)\meet(y-x)\meet(x \bigtriangleup (\UN-y))=\ZERO\big]$
    is the variety generated by the co-Heyting algebra $\Lcinq$
    (see figure~\ref{fi:Lcinq}).
  \item
    $\cV_5=\cV_2+\big[\big(((\UN-x)\meet x)-y\big)\meet y=\ZERO\big]$ 
    is the variety generated by $\Ltrois$ (see figure~\ref{fi:Lcinq}). 
  \item
    $\cV_6=\cV_1+\big[(x-y)\meet(y-x)=\ZERO\big]$
    is the variety generated by the chains. 
  \item
    $\cV_7=\cV_1+\big[(1-x)\meet x=\ZERO\big]$ is the variety of boolean
    algebras (which are exactly the co-Heyting algebras of dimension 
    $\leq 0$). 
  \item
    $\cV_8$ is the trivial variety $\cV_1+\big[\UN=\ZERO\big]$ reduced
    to $\Lun$ (see figure~\ref{fi:Lcinq}).
\end{itemize}
Note that the product of an empty family of
co-Heyting algebras is just $\Lun$.
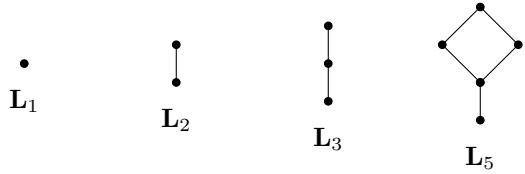
\begin{figure}[ht]
  \begin{center}
    \newcommand{\noir}[1]{\filldraw #1 circle (.1)}
    \begin{tikzpicture}[scale=.5]
      % L1
      \noir{(0,0)};
      \draw (0,-1) node{$\Lun$};

      % L2
      \draw (4,-.5) -- (4,.5);
      \noir{(4,-.5)};
      \noir{(4,.5)};
      \draw (4,-1.5) node{$\Ldeux$};

      % L3
      \draw (8,-1) -- (8,1);
      \noir{(8,-1)};
      \noir{(8,0)};
      \noir{(8,1)};
      \draw (8,-2) node{$\Ltrois$};
      
      % L5
      \begin{scope}[xshift=12cm,yshift=-1.5cm]
        \draw (0,0) -- (0,1) -- (-1,2) -- (0,3) -- (1,2) -- (0,1);
        \noir{(0,0)};
        \noir{(0,1)};
        \noir{(-1,2)};
        \noir{(0,3)};
        \noir{(1,2)};
        \noir{(0,1)};
        \draw (0,-1) node{$\Lcinq$};
      \end{scope}
    \end{tikzpicture}
    \caption{\label{fi:Lcinq} Four basic co-Heyting algebras}
  \end{center}
\end{figure}

\paragraph{Model-completion and super-intuitionistic logics.}

For an introduction to the basic notions of first-order
model-theory (language, formulas, elementary classes of structures,
models and existentially closed models of a theory) we
refer the reader to any introductory book, such as \cite{hodg-1997} or
\cite{chan-keis-1990}. 

Every model of a universal theory $T$ embeds in an existentially
closed model. If the class of all
existentially closed models of $T$ is elementary, then the
corresponding theory $\overline{T}$ is called the {\df model
companion} of $T$. The model companion eliminates quantifiers if and
only if $T$ has the amalgamation property, in which case
$\overline{T}$ is called the {\df model completion} of $T$. By abuse
of language we will speak of the model completion of a variety in
place of the model completion of the theory of this variety. 

It is an elementary fact that formulas in the first order
intuitionistic propositional calculus (IPC$^1$) can be considered as
terms (in the usual model-theoretical sense) in the language of
Heyting algebras, and formulas in the second order intuitionistic
propositional calculus (IPC$^2$) as first order formulas in the
language of Heyting algebras. In particular, if a variety of Heyting
algebras has a model completion then it appears, following
\cite{ghil-zawa-1997} that the corresponding
super-intuitionistic logic has the property that IPC$^2$ is
interpretable in IPC$^1$, in the sense of Pitts \cite{pitt-1992}. 

Finally let us recall the criterion for model completion which
makes the link with theorems~\ref{th:EC-intro} and
\ref{th:embed-intro}.

\begin{fact}\label{fa:mod-comp}
  A theory $\overline{T}$ is the model completion of a universal
  theory $T\subseteq\overline{T}$ if and only if it satisfies the two
  following conditions.
  \begin{enumerate}
    \item
      Every existentially closed model of $T$ is a model of
      $\overline{T}$. 
    \item 
      Given a model $L$ of $\overline{T}$, a finitely generated
      substructure $L_0$ of $L$ and a finitely generated model $L_1$
      of $T$ containing $L_0$, there is an embedding of $L_1$ into an
      elementary extension of $L$ which fixes $L_0$ pointwise. 
  \end{enumerate}
\end{fact}

\paragraph{The finite model property.} 

A variety $\cV$ of co-Heyting algebras has the finite model property
if any equation valid on every finite algebra in $\cV$ is valid on
every algebra of $\cV$. 

\begin{proposition}\label{pr:fmp}
  For a variety $\cV$ of co-Heyting algebras the following properties
  are equivalent. 
  \begin{enumerate}
    \item\label{it:fmp1}
      $\cV$ has the finite model property.
    \item\label{it:fmp2}
      For every quantifier-free $\ltc$\--formula $\varphi(x)$ and every
      algebra $L$ in $\cV$ such that $L\models\exists x\; \varphi(x)$, there exists a
      finite algebra $L'$ in $\cV$ such that $L'\models\exists x\; \varphi(x)$.
  \end{enumerate}
\end{proposition}

\begin{proof}
For every equation $\theta(x)$, if there is an algebra $L$ in $\cV$ on
which $\theta(x)$ is not valid then condition~\ref{it:fmp2} applied to
$\varphi(x)\equiv\lnot \theta(x)$ gives a finite algebra in $\cV$ on which $\theta(x)$ is not
valid. This proves that (\ref{it:fmp2})$\Rightarrow$(\ref{it:fmp1}). For the
converse, see \cite{darn-junk-2011} proposition~8.1. 
\end{proof}

The finite model property holds obviously for every locally finite
variety, but also\footnote{For example corollary 2.2.1 of
\cite{mcka-1968} applies to $\cV_2$, as well as to $\cV_1$.} for
$\cV_1$ and $\cV_2$. We combine it with a bit of model-theoretic
non-sense in the following lemmas. 

\begin{lemma}\label{le:non-sense-fin}
  Let $\cV$ be a variety of co-Heyting algebras having the finite
  model property. Let $\theta(x)$ and $\phi(x,y)$ be quantifier-free
  $\ltc$\--formulas. Assume that for every finite co-Heyting algebra
  $L_0$ and every tuple $a$ of elements of $L_0$ such that $L_0\models\theta(a)$, there
  exists an extension $L_1$ of $L_0$ which satisfies $\exists y\;\phi(a,y)$. Then
  every algebra existentially closed in $\cV$ satisfies the
  following axiom:
  \begin{displaymath}
    \forall x\;(\theta(x) \longrightarrow \exists y\;\phi(x,y)) 
  \end{displaymath}
\end{lemma}

\begin{proof}
Let $L$ be an existentially closed co-Heyting algebra which satisfies
$\theta(a)$ for some tuple $a$. Let $\Sigma$ be its quantifier-free diagram,
that is the set of all atomic and negatomic $\ltc$\--formulas (with
parameters) satisfied in $L$. Let $\Sigma_0$ be an arbitrary finite subset
of $\Sigma$. The conjunction of $\theta(a)$ and the elements of $\Sigma_0$ is a
quantifier-free formula (with parameters) $\Delta(a,b)$. Since 
$L\models\exists x,y\;\Delta(x,y)$ and $\cV$ has the finite model property, by
proposition~\ref{pr:fmp} there exists a finite co-Heyting algebra $L_0$ and a
tuple $(a_0,b_0)$ of elements of $L_0$ such that $L_0\models\Delta(a_0,b_0)$. In
particular $L_0\models\theta(a_0)$ hence by assumption $L_0$ admits an extension
$L_1$ which satisfies $\exists y\;\phi(a_0,y)$. So $L_1$ is a model of this
formula and of $\Sigma_0$ (because $\Sigma_0$ is quantifier free and already
satisfied in $L_0$). We have proved that the union of $\Sigma$ and $\exists
y\;\phi(a,y)$ is finitely satisfiable hence by the model-theoretic
compactness theorem, it admits a model $L'$ in which $L$ embeds
naturally. Since $L$ is existentially closed it follows that $L$
itself satisfies $\exists y\;\phi(a,y)$.
\end{proof}

\begin{lemma}\label{le:non-sense-prod}
  Let $\cV$ be a variety of co-Heyting algebras having the finite
  model property. Let $\theta'(x)$ and $\phi'(x,y)$ be $\ltc$\--formulas
  that are conjunctions of equations. Assume that:
  \begin{enumerate}
    \item
      There is a subclass $\cC$ of $\cV$ such that a finite co-Heyting
      algebra belongs to $\cV$ if and only if it embeds into the
      direct product of a finite (possibly empty) family of algebras
      in $\cC$. 
    \item 
      For every algebra $L$ in $\cC$ and every $a=(a_1,\dots,a_m)\in L^m$
      such that $L\models\theta'(a)$ there is an extension $L'$ of $L$ in $\cV$
      and some $b=(b_1,\dots,b_n)\in L'^n$ such that $L'\models \phi'(a,b)$. If
      moreover $a_1\neq\ZERO$ then one can require all the $b_i$'s to be
      non zero.
  \end{enumerate}
  Then every algebra existentially closed in $\cV$ satisfies:
  \begin{displaymath}
    \forall x\;\left[\left(\theta'(x)\conj x_1\neq\ZERO\right) \to \exists y\;\left(\phi'(x,y)\conj \cconj_{i\leq
    n}y_i\neq\ZERO\right)\right]
  \end{displaymath}
\end{lemma}

Although somewhat tedious, this lemma will prove to be helpful for the
varieties $\cH_2$, $\cH_4$, $\cH_5$ and $\cH_6$.

\begin{proof}
Let $L$ be a finite algebra in $\cV$ and $a=(a_1,\dots,a_m)\in L^m$. Assume
that $L\models\theta'(a)\conj a_1\neq\ZERO$. 

By assumption there are $L_1,\dots,L_r$ in $\cC$ such that $L$ embeds
into the direct product of the $L_j$'s. So each $a_i$ can be
identified with $(a^1_i,\dots,a^r_i)\in L_1\times\cdots\times L_r$. 

For every $j \leq r$ let $a^j=(a^j_1,\dots,a^j_m)\in L_j^m$. Since $\theta'(x)$ is
a conjunction of equations and $L\models\theta'(a)$, we have $L_j\models\theta'(a^j)$. Thus by
assumption there is an extension $L'_j$ of $L_j$ and a tuple
$b^j=(b^j_1,\dots,b^j_n)\in L'^n_j$ such that $L'_j\models\phi'(a^j,b^j)$. Moreover
if $a^j_1\neq\ZERO$ then we do require $b^j_i\neq\ZERO$ for every $i \leq n$. 

Let $L'$ be the direct product of the $L'_j$'s. For every $i \leq n$ let
$b_i=(b^1_i,\dots,b^r_i)$ and $b=(b_1,\dots,b_n)\in L'^n$. The algebra $L'$ is an
extension of $L$ in $\cV$, and since $\phi'(x,y)$ is a conjunction of
equation by construction $L'\models\phi'(a,b)$. 

Moreover $a_1=(a^1_1,\dots,a^r_1)$ is non zero, so there is an index $j \leq
r$ such that $a^j_1 \neq 0$. Then by construction for every $i\leq n$, 
$b^j_i\neq\ZERO$ hence $b_i$ is non zero. 

So we can apply lemma~\ref{le:non-sense-fin} to the variety $\cV$ with
the quantifier free formulas $\theta(x)$ and $\phi(x,y)$ defined by:
\begin{displaymath}
  \theta(x)\equiv\theta'(x)\conj x_1\neq\ZERO\quad\mbox{ and }\quad
  \phi(x,y)\equiv\phi'(x,y)\conj \cconj_{i\leq n}y_i\neq\ZERO
\end{displaymath}
\end{proof}

\paragraph{Decreasing subsets and finite duality.}

For any element $a$ and any subset $A$ of an ordered set $E$
we let: 
\begin{displaymath}
  a^\downarrow=\{b \in X\tq b \leq a\}\qquad\mbox{and}\qquad A\downarrow=\bigcup_{a \in A}a^\downarrow
\end{displaymath}
A {\df decreasing subset} of
$E$ is a subset such that $A=A\downarrow$. The family $\dec(E)$ of all decreasing subsets
of $E$ are the closed sets of a topology on $E$, hence a co-Heyting
algebra with operations:
\begin{displaymath}
  A \join B = A \cup B\qquad A \meet B = A \cap B\qquad A-B=(A\setminus B)\downarrow
\end{displaymath}
Its completely join irreducible elements are precisely the 
decreasing sets $x^\downarrow$ for $x$ ranging over $E$. 

It is folklore that if $L$ is a finite co-Heyting algebra, then the
map $\iota_L:a\mapsto a^\downarrow\cap\jirr(L)$ is an $\ltc$\--isomorphism from $L$ to the
family $\dec(\jirr(L))$ of all decreasing subsets of $\jirr(L)$, whose
inverse is the map $A\mapsto\jjoin A$. This actually defines a contravariant
duality between the category of finite co-Heyting algebras and the
category of finite posets with morphisms the increasing maps $\psi:I'\to I$
satisfying the property:
\begin{description}
  \item[Up:] $\forall x'\in I'$, $\psi(x'^\uparrow)\subseteq \psi(x')^\uparrow$.
\end{description}
This property can be rephrased as: 
\begin{displaymath}
  \forall x'\in I',\ \forall x\in I\ \big[ \psi(x')\leq x \Rightarrow 
  \exists y'\in I',\ \big( x'\leq y' \mbox{ and }x=\psi(y') \big)\big]
\end{displaymath}
For the lack of a reference we provide here a brief overview of this
duality, which can be seen as a restriction of the classical Stone's
duality between distributive lattices and their prime filter spectrum.
The point is that condition Up characterizes among increasing maps
(which are exactly the spectral maps, since we restrict to finite
spectral spaces) those which come from an $\ltc$\--morphism. 
\begin{itemize}
  \item
    Given a map $\psi:I'\to I$ as above we consider $[\psi]:\dec(I)\to \dec(I')$
    defined by $[\psi](A)=\psi^{-1}(A)$. Note that $\psi^{-1}(A)$ is indeed a
    decreasing subset of $I'$ because $A$ itself is decreasing and $\psi$
    is an increasing map. 
  \item 
    Conversely, given an $\ltc$\--morphism $\varphi:L\to L'$ of finite
    co-Heyting algebras, we consider $[\varphi]:\jirr(L')\to\jirr(L)$ defined
    by $[\varphi](x')=\min(\varphi^{-1}(x'^\uparrow))$. Note that $x'^\uparrow$ is a prime
    filter of $L'$ hence $\varphi^{-1}(x'^\uparrow)$ is a prime filter of $L$,
    which ensures that its generator $\min(\varphi^{-1}(x'^\uparrow))$ is indeed a
    join irreducible element of $L$. 
\end{itemize}

\begin{fact}\label{fa:dual-I2L}
  Let $\psi:I'\to I$ be an increasing map between finite posets
  satisfying condition Up. Then $[\psi]:\dec(I)\to\dec(I')$ is an
  $\ltc$\--morphism. Moreover $\psi$ is surjective if and only if $[\psi]$
  is injective.
\end{fact}

\begin{fact}\label{fa:dual-L2I}
  Conversely, let $\varphi:L\to L'$ an $\ltc$\--morphism between finite
  co-Heyting algebras. Then $[\varphi]:\jirr(L')\to\jirr(L)$ is an increasing
  map which satisfies condition Up. Moreover $\varphi$ is injective if and
  only if $[\varphi]$ is surjective.
\end{fact}

The proofs are good exercises that we leave to the reader. On the
other hand we provide a self-contained proof of the next proposition
which summarises the only parts of this duality that we will
use\footnote{We will actually use facts~\ref{fa:dual-I2L} and
\ref{fa:dual-L2I} in the proof of the implication
(\ref{it:card})$\Rightarrow$(\ref{it:min}) of corollary~\ref{co:mini-prim}. But
from this corollary we will only use the reverse implication
(\ref{it:prim})$\Rightarrow$(\ref{it:card}), which does not require these facts.}.
Since it provides a flexible tool to build an extension of a finite
co-Heyting algebra with prescribed
conditions, it will play a central role in our constructions. 

\begin{proposition}\label{pr:finite-injection}
  Let $L$ be a finite co-Heyting algebra and $\cI$ an ordered set.
  Assume that there is a surjective increasing map $\pi$ from $\cI$ onto
  $\jirr(L)$ which satisfies condition Up.
  Then there exists an $\ltc$\--embedding $\varphi$ of $L$ into
  $\dec(\cI)$ such that\footnote{Note that the compositum $\pi\circ\varphi$ is
  not defined. In this proposition $\varphi(a)$ is a decreasing subset of
  $\cI$ and $\pi(\varphi(a))=\{\pi(\xi)\tq\xi\in\varphi(a)\}$.} $\pi(\varphi(a))=a^\downarrow\cap\jirr(L)$ for
  every $a \in L$. 
\end{proposition}

\begin{proof}
Let $L^\downarrow$ denote $\dec(\jirr(L))$ and $L'^\downarrow$ denote $\dec(\cI)$.
The map $\varphi^\downarrow:L^\downarrow \to L'^\downarrow$ defined by $\varphi^\downarrow(A)=\pi^{-1}(A)$ is clearly a
morphism of bounded lattices ($\varphi^\downarrow(A) \in L'^\downarrow$ for every $A \in L^\downarrow$
because $\pi$ is increasing) which is injective because $\pi$ is
surjective. Clearly $\pi(\varphi^\downarrow(A))=A$ for every $A \in L^\downarrow$ so it remains to
check that $\varphi^\downarrow$ is an
$\ltc$\--embedding of $L^\downarrow$ into $L'^\downarrow$ in order to conclude that
$\varphi:a\mapsto\varphi^\downarrow(a^\downarrow\cap\jirr(L))$ is the required $\ltc$\--embedding of $L$ into
$\dec(\cI)$. 

Let $A,B \in L^\downarrow$ and choose any $x'\in \varphi^\downarrow(A-B)=\pi^{-1}\big(( A\setminus B)\downarrow\big)$. 
Then $\pi(x')\leq x$ for some $x\in A\setminus B$. Condition Up gives 
$y \in \cI$ such that $\pi(y)=x$ and $x' \leq y$ so $x' \in \pi^{-1}(A\setminus B)\downarrow$.
\begin{displaymath}
  \pi^{-1}(A\setminus B)\downarrow = \big(\pi^{-1}(A) \setminus \pi^{-1}(B)\big)\downarrow
                = \big(\varphi^\downarrow(A) \setminus \varphi^\downarrow(B)\big)\downarrow 
                = \varphi^\downarrow(A)-\varphi^\downarrow(B)
\end{displaymath}
We conclude that $\varphi^\downarrow(A-B)\subseteq\varphi^\downarrow(A)-\varphi^\downarrow(B)$. The reverse inclusion is
immediate since $\varphi^\downarrow(A)\subseteq\varphi^\downarrow\big((A-B)\cup B)\big)=\varphi^\downarrow(A-B)\cup\varphi^\downarrow(B)$
implies that $\varphi^\downarrow(A)-\varphi^\downarrow(B)\subseteq\varphi^\downarrow(A-B)$. So $\varphi^\downarrow(A-B)=\varphi^\downarrow(A)-\varphi^\downarrow(B)$.
\end{proof}

\section{Minimal finite extensions}
\label{se:minim-exten}

This section is devoted to the study of minimal finite proper
extensions of a finite co-Heyting algebra $L_0$. We are going to show
(see remark~\ref{re:corr-sign-ext}) that they are in one-to-one
correspondence with {\df signatures} in $L_0$, that is triples
$(g,H,r)$ such that $g$ is a join irreducible element of $L_0$,
$H=\{h_1,h_2\}$ is a set of one or two elements of $L_0$ and:
\begin{itemize}
  \item
    either $r=1$ and $h_1=h_2<g$;
  \item 
    or $r=2$ and $h_1 \join h_2$ is the unique predecessor of $g$ (both
    possibilities, $h_1=h_2$ and $h_1 \neq h_2$ may occur in this case,
    see example~\ref{ex:prim-ext-sgn}).
\end{itemize}
Let $L$ be an $\ltc$\--extension of $L_0$ and $x\in L$. We introduce the
following notation.
\begin{itemize}
\item
For every $a\in L_0$, $a^-=\jjoin\{b\in L_0\tq b<a\}$.
\item
$L_0\gen x$ denotes the $\ltc$\--substructure of $L$ generated by
$L_0\cup\{x\}$.
\item
$g(x,L_0)=\mmeet\{a\in L_0\tq x\leq a\}$. 
\end{itemize}
Clearly $a\in\jirr(L_0)$ if and only if $a^-$ is the unique predecessor
of $a$ in $L_0$ (otherwise $a^-=a$). We say that a tuple $(x_1,x_2)$
of elements of $L$ is {\df primitive over $L_0$} if they are
both\footnote{The other conditions imply that $x_1$ and $x_2$ do not
belong to $L_0$, provided they are both non zero.} not
in $L_0$ and there exists $g\in\jirr(L_0)$ such that: 
\begin{description}
  \item[P1:]
  $g^-\meet x_1$ and $g^-\meet x_2$ belong to $L_0$.
  \item[P2:]
  One of the following holds: 
  \begin{enumerate}
    \item
      $x_1=x_2$ and $g^-\meet x_1\ll x_1 \ll g$.
    \item
      $x_1\neq x_2$ and $x_1\meet x_2\in L_0$, $g-x_1=x_2$, $g-x_2=x_1$.
  \end{enumerate}
\end{description}

\begin{remark}\label{re:g-prim}
  If $(x_1,x_2)$ is a primitive tuple over $L_0$, and $g\in\jirr(L_0)$
  satisfies P1 and P2, then $g=g(x_1,L_0)=g(x_2,L_0)$. Indeed, $g(x_i,
  L_0) \leq g$ because $x_i \leq g$. On the other hand $g(x_i,L_0) \not< g$
  since otherwise $g(x_i,L_0) \leq g^-$ hence {\it a fortiori} $x_i\leq g^-$
  and finally $x_i = x_i \meet g^- \in L_0$, a contradiction. 
\end{remark}

If an extension $L$ of $L_0$ is generated over $L_0$ by a primitive
tuple $(x_1,x_2)$ we call it a {\df primitive extension}. By the above
remark, $(g,\{g^- \meet x_1,g^- \meet x_2\}, \card\{x_1,x_2\})$ is then a signature in
$L_0$ which is determined by $(x_1,x_2)$ (actually by any of $x_1$, $x_2$).
We call it the {\df signature of the tuple $(x_1,x_2)$} (or simply of
$x_1$).

\begin{example}\label{ex:prim-ext-sgn}
  The following primitive extensions and their elements $a,b,c\dots$ are
  shown in figure~\ref{fi:prim-ext}.
  \begin{itemize}
    \item $\Ldeux\subset\Ltrois$:
      $(a,a)$ is primitive over $\Ldeux$ with signature
      $(\UN,\{\ZERO\},1)$. 
    \item $\Ldeux\subset\Ldeux\times\Ldeux$:
      $(a,b)$ is a primitive over $\Ldeux$, with signature
      $(\UN,\{\ZERO\},2)$.
    \item $\Lcinq^*\subset\Lneuf$:
      $(c,d)$ is primitive over $\Lcinq^*$ with signature
      $(\UN,\{a,b\},2)$. 
  \end{itemize}
\end{example}

\begin{figure}[ht]
  \begin{center}
    \newcommand{\noir}[1]{\filldraw #1 circle (.1)}
    \newcommand{\blanc}[1]{\draw[fill=white] #1 circle (.1)}

    \begin{tikzpicture}[scale=.5]

      % Figure de gauche
      \draw (0,-1) -- (0,0) node[left]{$a$} -- (0,1);
      \noir{(0,0)};
      \blanc{(0,-1)};
      \blanc{(0,1)};
      \draw (0,-3) node{$\Ldeux\subset\Ltrois$};

      % Figure du centre
      \begin{scope}[xshift=5cm]
        \draw (-1,0) node[left]{$a$}-- (0,1) 
          -- (1,0) node[right]{$b$} -- (0,-1) -- cycle;
        \noir{(-1,0)};
        \noir{(1,0)};
        \blanc{(0,-1)};
        \blanc{(0,1)};
        \draw (0,-3) node{$\Ldeux\subset\Ldeux\times\Ldeux$};
      \end{scope}

      % Figure de droite
      \begin{scope}[xshift=12cm]
        \begin{scope}[every node/.style={draw,shape=circle,fill=white,
                 inner sep=0mm, minimum size=1mm}]
          \foreach \k in {0,1,2}
            { \draw (\k,\k-2) node{} -- ++(-1,1) node{} -- ++(-1,1) node{};
              \draw (-\k,\k-2) node{} -- ++(1,1) node{} -- ++(1,1) node{}; }
        \end{scope}
        \begin{scope}[every node/.style={draw,shape=circle,fill=black,
                 inner sep=0mm, minimum size=1mm}]
          \draw (-2,0) node(c){} (-1,1) node{} 
            (1,1) node{} (2,0) node(d){};
        \end{scope}
        \draw (-1,-1) node[below left]{$a$}
          (1,-1) node[below right]{$b$}
          (c) node[below left]{$c$}
          (d) node[below right]{$d$};
        \draw (0,-3) node{$\Lcinq^*\subset\Lneuf$};
      \end{scope}

    \end{tikzpicture}

    \caption{\label{fi:prim-ext} 
      Three examples of primitive extensions $L_0\subset L$. \\
      Inside $L$, the white points represent $L_0$.}
  \end{center}
\end{figure}
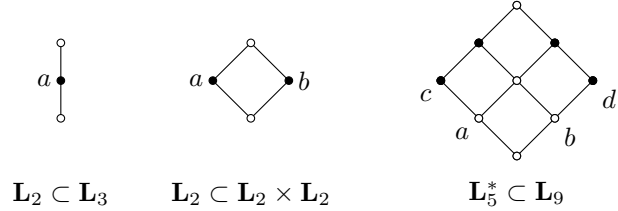
Dually, the next figure shows the projections of $\jirr(L)$ onto
$\jirr(L_0)$ corresponding to each of the three embeddings $L_0\subset L$ in
figure~\ref{fi:prim-ext}.

\begin{figure}[ht]
  \begin{center}
    \newcommand{\noir}[1]{\filldraw #1 circle (.1)}
    \begin{tikzpicture}[scale=.5]
      \draw (0,0) node[left]{$a$} -- (0,1.5) node[left]{$\UN$};
      \draw (0,-2) node[below]{$\UN$};
      \draw[->>,color=gray] (0,-.5) -- (0,-1.5);
      \noir{(0,-2)};
      \noir{(0,0)};
      \noir{(0,1.5)};
      \begin{scope}[xshift=5.5cm]
        \draw (0,-2) node[below]{$\UN$};
        \draw (-.7,.75) node[above]{$a$};
        \draw (.7,.75) node[above]{$b$};
        \draw (0,.75) node[below]{$\color{gray}\underbrace{\qquad\quad}$};
        \draw[->>,color=gray] (0,-.2) -- (0,-1.5);
        \noir{(0,-2)};
        \noir{(-.7,.75)};
        \noir{(.7,.75)};
      \end{scope}
      \begin{scope}[xshift=12cm]
        \draw (-.8,-2) node(aa)[below]{$a$};
        \draw (.8,-2) node(bb)[below]{$b$};
        \draw (0,-1.2) node(UN)[below]{$\UN$};
        \draw (-.8,0) node(a)[left]{$a$};
        \draw (.8,0) node(b)[right]{$b$};
        \draw (-.8,1.5) node(c)[left]{$c$};
        \draw (.8,1.5) node(d)[right]{$d$};
        \draw (0,1.5) node[below]{$\color{gray}\underbrace{\qquad\quad}$};
        \draw (-.8,0) -- (-.8,1.5) (.8,0) -- (.8,1.5) (-.8,-2) -- (0,-1.2) -- (.8,-2);
        \draw[->>,color=gray] (-.8,-.5) -- (-.8,-1.5);
        \draw[->>,color=gray] (.8,-.5) -- (.8,-1.5);
        \draw[->>,color=gray] (0,.5) -- (0,-.8);
        \noir{(-.8,-2)};
        \noir{(-.8,0)};
        \noir{(-.8,1.5)};
        \noir{(.8,-2)};
        \noir{(.8,0)};
        \noir{(.8,1.5)};
        \noir{(0,-1.2)};
      \end{scope}
    \end{tikzpicture}
  \end{center}
  \caption{Dual projections.}
\end{figure}
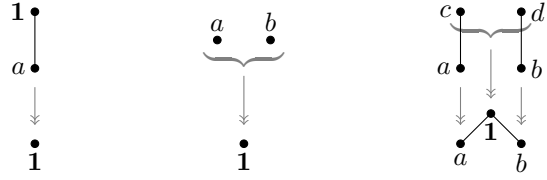

\begin{theorem}\label{th:ext-primitive}
Let $L_0$ be a finite co-Heyting algebra, $L$ an extension generated
over $L_0$ by a primitive tuple $(x_1,x_2)$, and let $g=g(x_1,L_0)$.

Then $L$ is exactly the upper semi-lattice generated over $L_0$ by
$x_1,x_2$. It is a finite co-Heyting algebra and one of the
following holds: 
\begin{enumerate}
\item
\label{item:ext prim 1 gene}
$x_1=x_2$ and $\jirr(L)=\jirr(L_0)\cup\{x_1\}$.
\item
\label{item:ext prim 2 gene}%
$x_1\neq x_2$ and $\jirr(L)=\left(\jirr(L_0)\setminus\{g\}\right)\cup\{x_1,x_2\}$.
\end{enumerate}
\end{theorem}

\begin{proof}
Let $L_1$ be the upper semi-lattice generated over $L_0$ 
by $x_1,x_2$. In order to prove that $L_1=L$ it is sufficient 
to show that $L_1$ is an $\ltc$\--substructure of $L$. 

We first check that $L_1$ is a sublattice of $L$. 
Any two elements in $L_1$ can be written $a \join y$ and 
$a' \join y'$ with $a, a'$ in $L_0$ and $y,y'$ in $\{\ZERO,x_1,x_2\}$, so
\begin{displaymath}
  (a\join y)\meet(a'\join y')=(a\meet a')\join(a\meet y')\join(a'\meet y)\join(y\meet y').
\end{displaymath}
It suffices to show that each of the four elements joined in the
right-hand side belong to $L_1$. That $a\meet a'\in L_0$ is clear. 
If $a\geq y'$ then obviously $a\meet y'=y'\in L_1$. Otherwise, $a\ngeq y'$ implies
that $y'=x_i$ for some $i\in\{1,2\}$ and that $a\ngeq g$, hence $a\meet
g=a\meet g^-$. Because $g^-\meet x_i\in L_0$ by definition of a primitive tuple, we have
\begin{displaymath}
  a\meet x_i=a\meet(g\meet x_i)= (a\meet g^-)\meet x_i=a\meet(g^-\meet x_i)\in L_0.
\end{displaymath}
This proves that $a\meet y'\in L_1$, and by symmetry $a'\meet y\in L_1$. Finally,
if $y\leq y'$ or $y'\leq y$ then obviously $y\meet y'\in L_1$. Otherwise we have
$\{y,y'\}=\{x_1,x_2\}$ and $x_1\neq x_2$, hence $x_1\meet x_2\in L_0$ by
definition of a primitive tuple hence $y\meet y'\in L_1$ in this case too.
\smallskip

So $L$ is a sublattice of $L$. We turn now to the $-$ operation. 
\begin{eqnarray*}
  (a \join y) - (a' \join y') 
  &=& \left[a-(a' \join y')\right] \join \left[ y-(a'\join y') \right] \\
  &=& \left[ (a-a')-y' \right] \join \left[ (y-y')-a' \right]
\end{eqnarray*}
So it suffices to prove that $(a-a')-y'$ and $(y-y')-a'$ belong to
$L_1$. For the first one, we have $a-a'\in L_0$ hence it suffices to
check that $b-x_1\in L_1$ for every $b\in L_0$ (that $b-x_2\in L_1$ will
follow by symmetry). For the second one, note that $y-y'$ is either
$\ZERO, x_1$ or $x_2$. This is obvious if $y=y'$ or $y=\ZERO$ or
$y'=\ZERO$, and otherwise we have $\{y,y'\}=\{x_1,x_2\}$ and $x_1\neq x_2$.
But then, by definition of a primitive tuple
\begin{displaymath}
  x_1-x_2=(x_1\join x_2)-x_2=g-x_2=x_1
\end{displaymath}
and symmetrically $x_2-x_1=x_2$. Thus, in order to prove that 
$(y-y')-a'\in L_1$ it suffices to check that $x_1-b\in L_1$ for every
$b\in L_0$ (that $x_2-b\in L_1$ will follow by symmetry).

Let $b$ be any element of $L_0$. We first check that $b-x_1 \in L_1$. 
\begin{eqnarray*}
  b-x_1 &=& \left[(b \meet g) \join (b-g)\right]-x_1\\
        &=& \left[(b \meet g)-x_1\right] \join \left[(b-g)-x_1\right] \in L_1
\end{eqnarray*}
If $g \leq b$ then $(b \meet g) - x_1=g-x_1$ is either $g$ or $x_2$. If 
$g \nleq b$ then $b \meet g \leq g^-$ hence $b \meet g \meet x_1$ belongs to $L_0$ hence
so does $(b \meet g) - x_1=(b \meet g)-(b\meet g\meet x_1)$. So in any case $(b \meet g)
-x_1$ belongs to $L_1$. On the other hand $(b-g)-x_1=b-g$ belongs to
$L_0$, so finally $b-x_1\in L_1$. 

Now we check that $x_1 -b \in L_1$. This is clear if $x_1 \leq b$.
If $x_1 \nleq b$ then $g \nleq b$ hence $x_1 \meet b \leq g^-$. It follows that:
\begin{displaymath}
   x_1 \meet b \leq g^-\meet x_1 \ll x_1
\end{displaymath}
Indeed if $x_1=x_2$ then $g^-\meet x_1 \ll x_1$ by assumption, and if $x_1\neq
x_2$ then $g^-\meet x_1 < g$ and $g^-\meet x_1\in L_0$ hence $g^-\meet x_1 \ll g$
because $g$ is join irreducible. So $x_1-b=x_1-(x_1\meet b)=x_1$ belongs to $L_1$.
\smallskip

This proves that $L_1$ is an $\ltc$\--substructure of $L$. Since $L_1$
contains $L_0$ and $\{x_1,x_2\}$ it follows that $L_1=L$. 
\smallskip

We turn now to the description of $\jirr(L)$. Since $L_0$ is finite
and $L$ is generated by $L_0\cup\{x_1,x_2\}$ as an upper semi-lattice, it
follows immediately that $L$ is finite and:
\begin{equation}
\label{eq:irred L et L0: inclusion}
\jirr(L)\subseteq \jirr(L_0)\cup\{x_1,x_2\}
\end{equation}
If $x_1\neq x_2$ then of course $g=x_1\join x_2\notin\jirr(L)$. 
Conversely if $x_1=x_2$ then $x_1\ll g$ by definition of a primitive
tuple, so $g-x_1=g$. Then for any $a\in L_0$ we have $g-(x_1\cup a)=g-a$ is
equal to $\ZERO$ or $g$ by join irreducibility of $g$ in $L_0$, which
proves that $g\in\jirr(L)$. So:
\begin{equation}
\label{eq:irred L et L0: g}
g\in\jirr(L) \iff x_1= x_2
\end{equation}

Of course we cannot have $\jirr(L)=\jirr(L_0)$ since $L$ is a proper
extension of $L_0$ generated by $\jirr(L)$. So by (\ref{eq:irred L et
L0: inclusion}) and (\ref{eq:irred L et L0: g}) it only remains to
check that 
\begin{equation}\label{eq:irred L et L0: deuxieme inclusion}
\jirr(L_0)\setminus\{g\}\subseteq \jirr(L)
\end{equation}
Assume that $\jirr(L_0)\not\subseteq \jirr(L)$. 
Let $b\in\jirr(L_0)\setminus\jirr(L)$ 
and let $y_1,\dots,y_r$ ($r\geq 2$) be its $\join$\--irreducible 
components in $L$. By (\ref{eq:irred L et L0: inclusion}), 
each $y_i$ either belongs to $L_0$ or to $\{x_1,x_2\}$, 
and at least one of them does not belong to $L_0$. 
We may assume without loss of generality that $y_1=x_1$. 
Then $x_1\leq b$ hence $g\leq b$. If $g<b$ then $g\ll b$ 
since $b\in\jirr(L_0)$ so $b-g=b$, but then we have a contradiction: 
\begin{displaymath}
  y_1\leq b-g\leq b-x_1=\jjoin_{i=2}^r y_i 
\end{displaymath}
This proves that either $\jirr(L_0)\subseteq \jirr(L)$ or
$\jirr(L_0)\setminus\jirr(L)=\{g\}$, hence 
(\ref{eq:irred L et L0: deuxieme inclusion}) holds true in any case. 
\end{proof}

\begin{corollary}\label{co:mini-prim}
  Let $L$ be a finite extension of a co-Heyting algebra $L_0$. The
  following assertions are equivalent. 
  \begin{enumerate}
    \item\label{it:min}
      $L$ is a minimal proper extension of $L_0$.
    \item\label{it:prim}
      $L$ is a primitive extension of $L_0$.
    \item\label{it:card}
      $\card(\jirr(L))=\card(\jirr(L_0))+1$.
  \end{enumerate}
  As a consequence every finite extension $L_0\subset L$ is the
  union of a tower of primitive extensions $L_0\subset L_1\subset\cdots\subset L_n=L$
  with $n=\card(\jirr(L))-\card(\jirr(L_0))$. 
\end{corollary}

\begin{proof}
(\ref{it:prim})$\Rightarrow$(\ref{it:card}): This follows directly from
theorem~\ref{th:ext-primitive}.
\smallskip

(\ref{it:card})$\Rightarrow$(\ref{it:min}): Let $L_1$ be a proper extension of $L_0$
contained in $L$. We have to prove that $L=L_1$. The inclusion maps
$\varphi_0:L_0\to L_1$ and $\varphi_1:L_1\to L$ induce surjective
increasing maps $[\varphi_1]:\jirr(L)\to\jirr(L_1)$ and
$[\varphi_0]:\jirr(L)\to\jirr(L_0)$. This implies that
$\card(\jirr(L))\geq\card(\jirr(L_1))\geq\card(\jirr(L_0))$. The second
inequality is strict, otherwise $[\varphi_0]$ is a bijection hence so is
$\varphi_0$, contrary to our assumption that $L_1$ is a proper extension of
$L_0$. But (\ref{it:prim}) then implies that
\begin{displaymath}
  \card(\jirr(L))=\card(\jirr(L_1))=\card(\jirr(L_0))+1 
\end{displaymath}
Thus $[\varphi_1]$ is a bijection, hence so is $\varphi_1$, that is $L=L_1$. 
\smallskip

(\ref{it:min})$\Rightarrow$(\ref{it:prim}): By minimality of $L$ it suffices to
prove that $L$ contains a primitive extension of $L_0$, that is to
find in $L$ a primitive tuple $(x_1,x_2)$ over $L_0$. In order to do
so, let us take any element $x$ minimal in $\jirr(L)\setminus L_0$. Observe that
if $y$ is any element of $L$ strictly smaller than $x$, then all the
$\join$\--irreducible components of $y$ in $L$ actually belong to $L_0$
(by minimality of $x$) so $y\in L_0$. 

Let $g=g(x,L_0)$. For every $a\in L_0$, if $a<g$ then $x \nleq a$ hence
$a\meet x<x$, so $a\meet x\in L_0$. It follows that $g^-\meet x\in L_0$. In particular
$g^- \neq g$ hence $g\in\jirr(L_0)$. 

Moreover $g^- \meet x < x$ since $x \notin L_0$, hence $g^- \meet x \ll x$ because
$x$ is join irreducible in $L$. So in the case when $x \ll g$ we have 
proved that $(x,x)$ is primitive over $L_0$. 

On the other hand, when
$x \not\ll g$ then $g-(g-x)=x$, indeed:
\begin{displaymath}
  g-(g-x)=\big(x\join(g-x)\big)-(g-x)=x-(g-x) 
\end{displaymath}
The last term is either $\ZERO$ or $x$ due to the join irreducibility
of $x$. But it cannot be $\ZERO$ since $x \leq g-x$ would imply that
$g=g-x$ hence $x \ll g$, a contradiction. 

Note that $x\nleq g-x$ implies also that $x\meet(g-x)<x$ hence $x\meet(g-x)\in L_0$.
So when $x \not\ll g$ we have proved that $(x,g-x)$ is primitive over
$L_0$. 
\end{proof}

\begin{corollary}\label{co:isom-sign}
  Let $L_1$, $L_2$ be two finite co-Heyting algebras both generated over a
  common subalgebra by a primitive tuple. If these tuples have the same
  signature in $L_0$ then they are isomorphic over $L_0$ (there exists
  an $\ltc$\--isomorphism from $L_1$ to $L_2$ which fixes $L_0$
  pointwise).
\end{corollary}

\begin{proof}
Assume that $L_i$ is generated over $L_0$ by a primitive tuple
$(x_{i,1},x_{i,2})$ for $i=1,2$ having the same signature $(g,H,r)$ in
$L_0$. By definition of $H$, changing if necessary the numbering of
the $x_{2,j}$'s we can assume that for $i=1,2$:
\begin{equation}\label{eq:g-moins}%
  g^-\meet x_{1,j} = g^- \meet x_{2,j}
\end{equation}
By definition of $r$, $x_{1,1}\neq x_{1,2}$ if and only if $x_{2,1}\neq
x_{2,2}$. By definition of $g$ and by
theorem~\ref{th:ext-primitive} there exists a (unique) bijection
$\sigma$ from $\jirr(L_1)$ to $\jirr(L_2)$ which fixes $\jirr(L_0)$
pointwise and maps each $x_{1,j}$ to $x_{2,j}$.
Condition~(\ref{eq:g-moins}) ensures that $\sigma$ preserves the order. 
This determines an
$\ltc$\--isomorphism $[\sigma]:\dec(\jirr(L_2))\to\dec(\jirr(L_1))$ which
fixes $\dec(\jirr(L_0))$ pointwise. Recall that for every co-Heyting
algebra $L$ we let $\iota_L:L\to\dec(\jirr(L))$ denote the canonical
isomorphism defined by $\iota(a)=a^\downarrow\cap\jirr(L)$. We finally get that
$\iota_{L_1}^{-1}\circ[\sigma]\circ\iota_{L_2}^{\phantom{-1}}$ is an isomorphism from $L_2$
to $L_1$ over $L_0$. 
\end{proof}

\begin{remark}\label{re:corr-sign-ext}
  Corollary~\ref{co:mini-prim} shows that every minimal finite proper
  extension of $L_0$ is generated by a primitive tuple $(x_1,x_2)$,
  which is unique (up to permutation) by
  theorem~\ref{th:ext-primitive}. This tuple has a signature, which is
  then entirely determined by the embedding of $L_0$ into $L_1$. So we
  will call it the {\df signature of $L_1$ in $L_0$}. Of course every
  other extension of $L_0$ isomorphic to $L_1$ over $L_0$ will have
  the same signature in $L_0$. Conversely,
  corollary~\ref{co:isom-sign} shows that this signature determines
  $L_1$ up to isomorphism over $L_0$. Finally, it will be shown in the
  next section~\ref{se:V1} that every signature in $L_0$ is the
  signature of a primitive extension of $L_0$ (see
  remark~\ref{re:exist-sign}). Altogether this proves that the minimal
  finite proper extensions of $L_0$ up to isomorphism over $L_0$ are
  in one-to-one correspondence with the signatures in $L_0$. 
\end{remark}

\section{Density and splitting in $\cV_1$}%
\label{se:V1}%

For the variety $\cV_1$ of all co-Heyting algebras we introduce the
following axioms D1 and S1.

\begin{description}
  \item{\bf [Density D1]} 
    For every $a,c$ such that $c \ll a \neq \ZERO$ there exists a non zero
    element $b$ such that:
     \begin{displaymath}
       c \ll b \ll a 
     \end{displaymath}
  \item{\bf [Splitting S1]} 
    For every $a,b_1,b_2$ such that $b_1 \join b_2 \ll a \neq \ZERO$ there exists non
    zero elements $a_1$ and $a_2$ such that:
    \begin{displaymath}
      \begin{array}{c} 
        a-a_2 = a_1 \geq b_1\\
        a-a_1 = a_2 \geq b_2\\
        a_1 \meet a_2 = b_1 \meet b_2 
      \end{array}
    \end{displaymath}
\end{description}

Note that $a = a_1 \join a_2$, so the second axioms allows to split $a$ in
two pieces $a_1$, $a_2$ along $b_1\meet b_2$ (so the name of
``splitting'').

\begin{lemma}\label{le:D1}
  Let $a,c$ be two elements of a finite co-Heyting algebra $L$.
  If $c \ll a \neq \ZERO$ then there exists a finite co-Heyting algebra $L'$
  containing $L$ and a non zero element $b$ in $L'$ such that: 
  \begin{displaymath}
    c \ll b \ll a 
  \end{displaymath} 
\end{lemma}

\begin{proof}
Let $a_1,\ldots,a_r$ be the join irreducible components of $a$.
The idea of the proof is to add a new $\join$-irreducible element $\alpha_i$
immediately below each $a_i$.
Let $\mathcal{I}$ be the set $\jirr(L)$ augmented by $r$ new elements
$\alpha_1,\ldots,\alpha_r$. Extend the order of $\jirr(L)$ to $\mathcal{I}$ as
follows. The $\alpha_i$'s are two by two incomparable, and for every
$x\in\jirr(L)$ and every $i \leq r$:
\begin{displaymath}
  x < \alpha_i \quad\Leftrightarrow\quad x < a_i
\end{displaymath}
\begin{displaymath}
  \alpha_i < x \quad\Leftrightarrow\quad a_i \leq x 
\end{displaymath}
For every $\xi \in \cI$ let:
\begin{displaymath}
  \pi(\xi)=\left\{
  \begin{array}{ll} 
    x   & \mbox{if }\xi=x\mbox{ for some }x \in L \\
    a_i & \mbox{if }\xi=\alpha_i\mbox{ for some }i \leq r
  \end{array}\right.
\end{displaymath}
This is an increasing projection of $\mathcal{I}$ onto $\jirr(L)$. For
every $\zeta \in \mathcal{I}$ and every $x \in \jirr(L)$ such that $\pi(\zeta)\leq x$
there exists $\xi \in \mathcal{I}$ such that $\pi(\xi)=x$ and $\zeta \leq
\xi$: simply take $\xi=x$. Thus proposition~\ref{pr:finite-injection}
gives an $\ltc$\--embedding $\varphi$ of $L$ into $\dec(\cI)$. 

Each join irreducible element $x$ of $L$ smaller than $c$ is strictly smaller
than some join irreducible component $a_i$ of $a$ because $c \ll a$. 
By construction $x < \alpha_i< a_i$ in $\cI$ hence $\varphi(x)< \alpha_i^\downarrow < \varphi(a_i)$.
These three elements of $L'$ are join irreducible hence $\varphi(x) \ll \alpha_i^\downarrow \ll \varphi(a_i)$. It
follows that: 
\begin{displaymath}
  \varphi(c) = \jjoin \{ \varphi(x)\tq x\in\jirr(L),\ x\leq c \}
       \ll \jjoin_{1\leq i\leq r} \alpha_i^\downarrow
       \ll \jjoin_{1\leq i\leq r} \varphi(a_i) = \varphi(a)
\end{displaymath}
So we can take $L'=\dec(\cI)$ and $b=\jjoin_{1\leq i\leq r} \alpha_i^\downarrow$.
\end{proof}

\begin{lemma}\label{le:S1}
  Let $a,b_1,b_2$ be elements of a finite co-Heyting algebra $L$.
  If $b_1 \join b_2 \ll a \neq \ZERO$ then there exists a finite co-Heyting
  algebra $L'$ containing $L$ and non zero elements $a_1$, $a_2$ in
  $L'$ such that:
  \begin{displaymath}
    \begin{array}{c}  
      a-a_2 = a_1 \geq b_1\\
      a-a_1 = a_2 \geq b_2\\
      a_1 \meet a_2 = b_1 \meet b_2 
    \end{array}
  \end{displaymath}
\end{lemma}

The idea of the proof uses geometric intuition. Imagine that there
exists an $\ltc$\--embedding $\varphi$ of $L$ into the co-Heyting algebra
$L(X)$ of all semi-algebraic closed subsets of some real
semi-algebraic set $X$. It can be proved actually that such an
embedding exists, and that moreover we can reduce to the case when
$\varphi(a)$ is equidimensional (that is its local dimension is the same at
every point). So $A = \varphi(a)$, $B_1 = \varphi(b_1)$ and $B_2 = \varphi(b_2)$ are
closed semi-algebraic subsets of $X$. Let $X_1=X \setminus (B_2 \setminus B_1)$ and
$X_2=X \setminus (B_1 \setminus B_2)$. Glue two copies $X'_1$, $X'_2$ of $X_1$ and
$X_2$ along $B_1 \cap B_2$. The result $X'$ of this glueing is a real
semi-algebraic set which projects onto $X$ in an obvious way.
Figure~\ref{fi:glueing} shows this construction when $A=X$. 

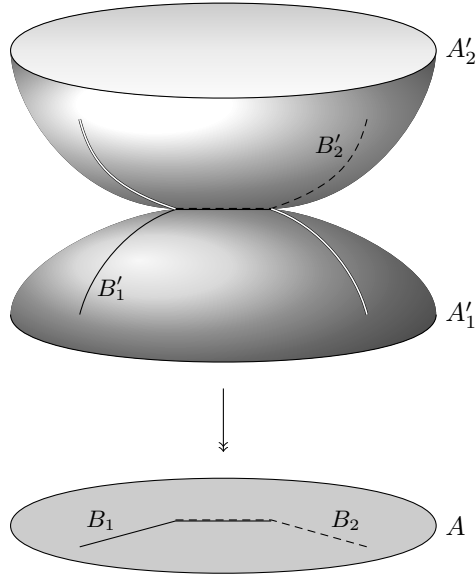
\begin{figure}[ht]
  \begin{center}

    \begin{tikzpicture}[scale=.7]

      % Partie haute
      \fill[ball color=white]%[top color=gray!20,bottom color=gray!80] 
        (-.8,0) .. controls +(-2,0) and +(0,-1) .. (-4,3)
        .. controls +(0,-1.2) and +(0,-1.2) .. (4,3)
        .. controls +(0,-1) and +(2,0) .. (.8,0)
        -- cycle;
      \filldraw[top color=gray!50,bottom color=white,middle color=gray!15]
        (-4,3) .. controls +(0,-1.2) and +(0,-1.2) .. (4,3)
        .. controls +(0,1.2) and +(0,1.2) .. (-4,3);

      % Partie basse
        \fill[ball color=gray!40]%[top color=gray!70,bottom color=gray!10] 
        (-.8,0) .. controls +(-1,0) and +(0,1) .. (-4,-2)
        .. controls +(0,-1.2) and +(0,-1.2) .. (4,-2)
        .. controls +(0,1) and +(1,0) .. (.8,0)
        -- cycle;
        \draw (-4,-2) .. controls +(0,-1.2) and +(0,-1.2) .. (4,-2);

      % B'1 et B'2
      \draw (-2.7,-2) .. controls +(.2,.8) and +(-.9,-.3) .. (-.9,-.01)
        -- (.9,-.01);
      \draw[densely dashed] (-.9,.01) -- (.9,.01)
        .. controls +(1.2,.4) and +(-.2,-.8) .. (2.7,1.7);
      \draw[double,very thin] 
        (2.7,-2) .. controls +(-.2,.8) and +(.9,-.3) .. (.9,0)
        (-.9,0) .. controls +(-1.2,.4) and +(.2,-.8) .. (-2.7,1.7);

      % A
      \begin{scope}[yshift=-6cm]
        \filldraw[fill=gray!40]
          (-4,0) .. controls +(0,-1.2) and +(0,-1.2) .. (4,0)
          .. controls +(0,1.2) and +(0,1.2) .. (-4,0);
        \draw (-2.7,-.4) -- (-.9,.09) -- (.9,.09);
        \draw[densely dashed] (-.9,.11) -- (.9,.11) -- (2.7,-.4);
      \end{scope}

      % L\'egende et fl\`eche
      \draw[->>] (0,-3.4) -- (0,-4.6);
      \draw (4,-6) node[right]{$A$}
        (4,-2) node[right]{$A'_1$}
        (4,3) node[right]{$A'_2$}
        (-4,3) node[left]{$\phantom{A'_2}$};
      
      {\small
        \draw (-2.3,-5.9) node{$B_1$}
          (2.3,-5.9) node{$B_2$}
          (-2.1,-1.5) node{$B'_1$}
          (2,1.2) node{$B'_2$};
      }

    \end{tikzpicture}
    \caption{\label{fi:glueing}%
      An example of glueing when $A=X$. \\
      The white curves represent cuts in $A'_1$ and $A'_2$.
    }%
  \end{center}
\end{figure}

This defines an embedding $L(X) \hookrightarrow L(X')$ which maps any semi-algebraic
subset $Y$ of $X$ closed in $X$ to the preimage $Y'$ of $Y$ via this
projection. Then $A'$ is the union of a copy of $A_1 = A \cap X_1$ and
$A_2 = A \cap X_2$ glued along $B_1 \cap B_2$. These copies $A'_1$, $A'_2$
of $A_1$ and $A_2$ are non empty semi-algebraic subsets of $X'$,
closed in $X'$, containing $B'_1$ and $B'_2$ respectively, such that: 
\begin{displaymath}
  A'_1 \cap A'_2 = B'_1 \cap B'_2
\end{displaymath}
The additional property that $A'_1$ (resp. $A'_2$)  is the topological
closure in $X'$ of $A' \setminus A'_2$ (resp. $A' \setminus A'_1$) then follows from
the assumption that $B_1 \cup B_2 \ll A$ and the fact that we reduced to
the case when $A'_1$ and $A'_2$ are equidimensional. 

\begin{proof}
The above geometric construction {\em is} a proof, provided an appropriate
dictionary between real semi-algebraic sets and elements of co-Heyting
algebras is given. However it would be longer to set explicitly this
dictionary than to hide the geometric intuition in a shorter
combinatorial proof. This is what we do now.

For each $x\in\jirr(L)$ such that $x \nleq b_2$ (resp $x \nleq b_1$) let
$\xi_{x,1}$ (resp. $\xi_{x,2}$) be a new symbol. For each $x\in\jirr(L)$ such that
$x \leq b_1 \meet b_2$ let $\xi_{x,0}$ be a new symbol. Let
$\mathcal{I}$ be the set of all these symbols and
define an order on $\mathcal{I}$ as follows: 
\begin{displaymath}
  \xi_{y,j} \leq \xi_{x,i} \quad\Leftrightarrow\quad y \leq x\hbox{ and }\{i,j\}\neq\{1,2\}
\end{displaymath}
The map $\pi:\xi_{x,i}\mapsto x$ defines an increasing projection of
$\cI$ onto $\jirr(L)$. For every $\zeta \in \cI$ and every $x \in \jirr(L)$
such that $\pi(\zeta)\leq x$ there exists $\xi$ such that $\pi(\xi)=x$ and $\zeta \leq \xi$.
Indeed if $\zeta=\xi_{y,1}$ then $\pi(\zeta)=y\nleq b_2$ so $x\nleq b_2$, hence $\xi_{x,1}$
exists and does the job. A symmetric argument applies if $\zeta=\xi_{y,2}$.
On the other hand if $\zeta=\xi_{y,0}$ then the existence of $\xi_{x,0}$ is
not guaranteed, but the existence of at least one among
$\xi_{x,0},\xi_{x,1},\xi_{x,2}$ is. Just take one of them. Thus
proposition~\ref{pr:finite-injection} gives an $\ltc$\--embedding $\varphi$
of $L$ into $\dec(\cI)$. For any $x \in \jirr(L)$ we have:
\begin{displaymath}
  \varphi(x) = \left\{
  \begin{array}{cl}
    \xi_{x,0}^\downarrow & \hbox{if $x \leq b_1 \meet b_2$}\\ 
    \xi_{x,1}^\downarrow & \hbox{if $x \leq b_1$ and $x \nleq b_2$} \\
    \xi_{x,2}^\downarrow & \hbox{if $x \leq b_2$ and $x \nleq b_1$} \\
    \xi_{x,1}^\downarrow \cup \xi_{x,2}^\downarrow & \hbox{otherwise.}
  \end{array} \right.
\end{displaymath}

Let $a_1,\ldots,a_r$ be the join irreducible components of $a$.
None of the $a_i$'s is smaller than $b_1$ or $b_2$ because by 
assumption $b_1 \join b_2 \ll a$, so each $\varphi(a_i)=\xi_{a_i,1}^\downarrow \cup \xi_{a_i,2}^\downarrow$. 
Define:
\begin{displaymath}
  \alpha_1=\bigcup_{1\leq i\leq r} \xi_{a_i,1}^\downarrow \quad\hbox{ and }\quad
  \alpha_2=\bigcup_{1\leq i\leq r} \xi_{a_i,2}^\downarrow 
\end{displaymath}
By construction $\varphi(a)-\alpha_1 = \alpha_2$ and $\varphi(a)-\alpha_2 = \alpha_1$ and both are non
empty since $r\geq 1$ (here we use that $a\neq\ZERO$). Moreover, for any
join irreducible element $x$ of $L$ such that $x \leq b_1$, we have $x \leq
a_j$ for some $j \leq r$. If $x\nleq b_2$, by definition of the order on
$\cI$ it follows that $\xi_{a_j,1}$ exists and $\xi_{x,1} \leq \xi_{a_j,1}$ hence:
\begin{displaymath}
  \varphi(x)=\xi_{x,1}^\downarrow \subseteq \xi_{a_j,1}^\downarrow \subseteq \alpha_1
\end{displaymath}
On the other hand if $x\leq b_2$ then $x\leq b_1\meet b_2$ so
$\xi_{x,0}$ exists. Since $\xi_{a_j,k}$ exists for some $k\in\{0,1,2\}$ and
$\xi_{x,0}\leq\xi_{a_j,k}$ we get:
\begin{displaymath}
  \varphi(x)=\xi_{x,0}^\downarrow \subseteq \xi_{a_j,k}^\downarrow \subseteq \alpha_1
\end{displaymath}
Thus in any case $\varphi(x)\subseteq\alpha_1$. It follows that $\varphi(b_1) \subseteq \alpha_1$, and
symmetrically $\varphi(b_2) \subseteq \alpha_2$.

It remains to check that $\alpha_1 \cap \alpha_2 = \varphi(b_1)\cap \varphi(b_2)$.
In order to do this, let $\xi$ be any element of $\cI$ and $x=\pi(\xi)$.
It is sufficient to prove that $\xi^\downarrow \subseteq \alpha_1 \cap \alpha_2$ if
and only if $\xi^\downarrow \subseteq \varphi(b_1) \cap \varphi(b_2)$ 

If $\xi^\downarrow \subseteq \varphi(b_1) \cap \varphi(b_2)$ then $x \leq b_1\meet b_2$ hence $\xi=\xi_{x,0}$ 
and $x \leq a_i$ for some $i\leq r$. It follows that $\xi_{x,0}\leq \xi_{a_i,1}$ so 
$\xi^\downarrow \subseteq \alpha_1$, and $\xi_{x,0}\leq \xi_{a_i,2}$ so $\xi^\downarrow \subseteq \alpha_1$. With other words 
$\xi^\downarrow \subseteq \alpha_1 \cap \alpha_2$. 

Conversely if $\xi^\downarrow \subseteq \alpha_1 \cap \alpha_2$ then there exists $i,j \leq r$ such that
$\xi^\downarrow \subseteq \xi_{a_i,1}$ and $\xi^\downarrow \subseteq \xi_{a_j,2}^\downarrow$. Thanks to the definition of
the ordering on $\mathcal{I}$ this implies that $\xi=\xi_{x,0}$ hence $x \leq
b_1 \meet b_2$ and so $\xi^\downarrow \subseteq \varphi(b_1) \cap \varphi(b_2)$. 
\end{proof}

\begin{theorem}\label{th:EC-V1}
  Every co-Heyting algebra existentially closed in $\cV_1$
  satisfies the density axiom D1 and the splitting axiom S1.
\end{theorem}

\begin{proof}
These two axioms can be written under the following form:
\begin{displaymath}
  \forall x\;(\theta(x) \longrightarrow \exists y\;\phi(x,y)) 
\end{displaymath}
where $\theta(x)$ and $\phi(x,y)$ are quantifier-free $\ltc$\--formulas. In
both cases we have shown in lemmas~\ref{le:D1} and 
\ref{le:S1} that for every finite co-Heyting algebra
$L$ and every tuple $a$ of elements of $L$ such that $L\models\theta(a)$, there
exists an extension $L'$ of $L$ which satisfies $\exists y\;\phi(a,y)$. 
The result follows, by lemma~\ref{le:non-sense-fin}.
\end{proof}

Here is a partial converse of theorem~\ref{th:EC-V1}.

\begin{theorem}\label{th:embed-V1}
  Let $L$ be a co-Heyting algebra satisfying the density axiom D1 and
  the splitting axiom S1. Let $L_0$ be a finite subalgebra of $L$. Let
  $L_1$ be a finite co-Heyting algebra containing $L_0$. Then there
  exists an embedding of $L_1$ into $L$ which fixes every point of
  $L_0$. 
\end{theorem}

\begin{proof}
By an immediate induction based on
corollary~\ref{co:mini-prim}, we reduce to the case when $L_1$
is generated over $L_0$ by a primitive tuple. Let $\sigma=(g,\{h_1,h_2\},r)$ be the
signature of $L_1$ in $L_0$. By corollary~\ref{co:isom-sign} it is sufficient to
prove that $\sigma$ is the signature of a primitive tuple of elements $x_1,
x_2 \in L$. 

{\it Case 1:} $r=1$ so $h_1=h_2$. Since $h_1 \leq g^- \ll g$, the splitting
property $S1$ applied to $g,g^-,h_1$ gives non zero elements $y_1, y_2$ in
$L$ such that:
\begin{displaymath}
  \begin{array}{c}
    g-y_1=y_2 \geq g^- \\
    g-y_2=y_1 \geq h_1 \\
    y_1 \meet y_2 = h_1
  \end{array}
\end{displaymath} 
We have $y_1 - h_1=(g-y_2)-h_1=(g-h_1)-y_2=g-y_2=y_1$ hence $h_1 \ll
y_1$. The density axiom D1 then gives $x \in L \setminus \{\ZERO\}$ such that $h_1
\ll x \ll y_1$. By construction:
\begin{displaymath}
  h_1 \leq g^- \meet x \leq y_2 \meet y_1 =h_1
\end{displaymath}
So $g^- \meet x =h_1 \in L_0$ and $g^-\meet x \ll x \ll g$, from which it follows
that $(x,x)$ is a primitive tuple with signature $(g,\{h_1\},1)=\sigma$ in $L_0$. 

{\it Case 2:} $r=2$ so $h_1 \join h_2 = g^-$. 
Since $g^- \ll g$ the splitting property S1 applied to $g,h_1,h_2$
gives non zero elements $y_1, y_2$ in $L$ such that:
\begin{displaymath}
  \begin{array}{c}
    g-y_1=y_2 \geq h_2 \\
    g-y_2=y_1 \geq h_1 \\
    y_1 \meet y_2 = h_1 \meet h_2 
  \end{array}
\end{displaymath}
We have $h_1 \leq g^- \meet y_1 = (h_1 \meet y_1) \join (h_2 \meet y_1) =h_1 \join (h_2 \meet
y_1)$. On the other hand $h_2 \meet y_1 \leq y_2 \meet y_1 = h_1 \meet h_2 \leq h_1$.
Therefore $g^- \meet y_1 =h_1$ and symmetrically $g^- \meet y_2 = h_2$ so both
of them belong to $L_0$. It follows that $(y_1,y_2)$ is a primitive
tuple with signature $(g,\{h_1,h_2\},2)=\sigma$ in $L_0$. 
\end{proof}

\begin{remark}\label{re:exist-sign}
  The above proof shows, incidentally, that any given signature in a
  finite co-Heyting algebra $L_0$ is the signature of an extension of
  $L_0$ generated by a primitive tuple (inside an existentially closed
  extension of $L_0$). 
\end{remark}

\begin{corollary}\label{co:embed-V1-fini} 
  If $L$ is a non-trivial co-Heyting algebra satisfying the axioms D1 and S1 then
  any finite non-trivial co-Heyting algebra embeds into $L$.
\end{corollary}

\begin{proof}
$\Ldeux$ is a common subalgebra of $L$ and any co-Heyting
algebra $L_1$. If $L_1$ is finite, theorem~\ref{th:embed-V1} applies
to $L_0=\Ldeux$, $L_1$ and $L$.
\end{proof}

\begin{corollary}\label{co:embed-V1-sat}
  If $L$ is a co-Heyting algebra satisfying the axioms D1 and S1,
  $L_0$ a finite subalgebra of $L$, and $L'$ any extension of $L_0$,
  then $L'$ embeds over $L_0$ into an elementary extension of $L$
  (or in $L$ itself if $L$ is sufficiently saturated). 
\end{corollary}

\begin{proof}
By standard model-theoretic argument, it suffices to show that any
existential formula with parameters in $L_0$ satisfied in $L'$ is
satisfied in $L$. Let $a$ be the list of all elements of $L_0$ and
$\Delta(a)$ be the conjunction of the quantifier free diagram of $L_0$, so
that a co-Heyting algebra is a model of the formula $\Delta(a)$ if and only
if $a$ enumerates a substructure isomorphic to $L_0$. 
Let $\exists x\;\theta(x,a)$ be any existential formula with parameters in $L_0$
satisfied in $L'$ (where $x$ is a tuple of variables). By
proposition~\ref{pr:fmp} there is a finite co-Heyting algebra $L_1$
satisfying $\exists x\;\theta(x,a)\land \Delta(a)$. Since $L_1$ models $\Delta(a)$ it contains
a subalgebra isomorphic to $L_0$, which we can then identify to $L_0$.
By corollary~\ref{co:embed-V1-fini}, $L_1$ embeds into $L$ over $L_0$
hence $L$ itself models $\exists x\;\theta(x,a)$ and the conclusion follows.
\end{proof}

\begin{remark}\label{re:embed-V1}
  If $L$ is a non-trivial co-Heyting algebra satisfying the axioms D1 and S1,
  then every non-trivial `co-Heyting algebra $L'$ embeds into an elementary extension of
  $L$ by corollary~\ref{co:embed-V1-sat} since $\Ldeux$ is a finite
  common subalgebra of $L'$ and $L$. 
\end{remark}

\section{Density and splitting in $\cV_2$}%

We introduce the following axioms:
\begin{description}
  \item{\bf [Density D2]} 
    Same as D1.
  \item{\bf [Splitting S2]} 
    Same as S1 with the additional assumption that $b_1 \meet b_2
    \meet (\UN-(\UN-a))=\ZERO$
\end{description}

\begin{fact}\label{fa:prod-V2}
  Let $L_0$ be a {\em finite} co-Heyting algebra. Let $x_1,\cdots,x_r$ be
  the join irreducible components of $\UN$ in $L_0$ (that is the
  maximal elements of $\jirr(L_0)$). The following conditions are
  equivalent:
  \begin{enumerate}
    \item\label{it:V2}%
      $L_0$ belongs to $\cV_2$.
    \item\label{it:V2-comp-irr}%
      $x_i \meet x_j = \ZERO$ whenever $i \neq j$.
    \item\label{it:V2-irr}%
      $L_0$ is isomorphic to a product of co-Heyting algebras
      $L_1,\dots,L_r$ such $\UN_{L_i}$ is join irreducible.
  \end{enumerate}
\end{fact}

This is folklore, but let us recall the argument. 

Clearly $\UN_{L_0}=\ZERO_{L_0}$ if and only if $r=0$, in which case
the whole fact is trivial. So let's assume that $r \geq 1$. 

(\ref{it:V2})$\Rightarrow$(\ref{it:V2-comp-irr})$\Leftarrow$(\ref{it:V2-irr}) is clear.
(\ref{it:V2})$\Leftarrow$(\ref{it:V2-comp-irr}) is an easy computation using
that $1-x$ is the join of all the join-irreducible components of $\UN$
which are not in $x^\downarrow$. (\ref{it:V2-comp-irr})$\Rightarrow$(\ref{it:V2-irr}) is
true because if we let $y_i=\jjoin_{j\neq i}x_j$ and $L_i=L/y_i^\downarrow$ for every
$i \leq r$, then it is an easy exercise to check that each $\UN_{L_i}$ is
join irreducible and to derive from (\ref{it:V2-comp-irr}) that the
natural map from $L$ to the product $L_1\times\cdots\times L_r$ is an isomorphism.

\begin{lemma}\label{le:D2}
  Let $L$ be a finite algebra in $\cV_2$ such that $\UN$ is join
  irreducible. Let $a,c$ be any two elements of $L$
  such that $c \ll a$. Then there exists an extension $L'$  of $L$ in 
  $\cV_2$ and an element $b$ in $L'$ such that: 
\begin{displaymath}
  c \ll b \ll a 
\end{displaymath} 
If moreover $a \neq \ZERO$ then one can require that $b \neq \ZERO$. 
\end{lemma}

\begin{proof}
By assumption $\UN$ has a unique predecessor $x$, thus $L_0=x^\downarrow$
has a natural structure of co-Heyting algebra. 

If $a=\ZERO$ one can take $b=\ZERO$. 

If $a=\UN$ then $c \leq x$. Let $L'$ be the co-Heyting algebra
obtained by inserting one new element $b$ between $x$ and $\UN$. 
Then $a$ and $b$ are join irreducible in $L'$ and $c < b < a$ hence 
we are done.

Otherwise $\ZERO \neq a \leq x$ thus lemma~\ref{le:D1} gives an
$\ltc$\--embedding $\varphi$ of $L_0$ into a co-Heyting algebra $L_1$
containing a non zero element $b$ such that $c \ll b \ll a$. Let $L'$ be
the co-Heyting algebra obtained by adding to $L_1$ a new element on
the top. The embedding $\varphi$ extends uniquely to an $\ltc$\--embedding
of $L$ into $L'$ and we are done. 
\end{proof}

\begin{lemma}\label{le:S2}
  Let $L$ be a finite algebra in $\cV_2$ such that $\UN$ is join
  irreducible. Let $a,b_1,b_2$ in $L$ be such that 
  $b_1 \join b_2 \ll a$ and $b_1 \meet b_2 \meet (\UN-(\UN-a))=\ZERO$. Then there
  exists an extension $L'$ of $L$ in $\cV_2$ and elements $a_1$, $a_2$
  such that:
\begin{displaymath}
  \begin{array}{c} 
    a-a_2 = a_1 \geq b_1\\
    a-a_1 = a_2 \geq b_2\\
    a_1 \meet a_2 = b_1 \meet b_2 
  \end{array}
\end{displaymath}
If $a \neq \ZERO$ one can require that $a_1,a_2$ are both non zero. 
\end{lemma}

\begin{proof}
By assumption $\UN$ has a unique predecessor $x$, thus $L_0=x^\downarrow$
has a natural structure of co-Heyting algebra. 

{\it Case~1:} $a=\ZERO$. One can take $a_1=a_2=\ZERO$. 

{\it Case~2:} $\ZERO \neq a \leq x$. Lemma~\ref{le:S1} gives an
$\ltc$\--embedding $\varphi$ of $L_0$ into a co-Heyting algebra $L_1$
containing non zero elements $a_1,a_2$ with the required properties.
Let $L'$ be the co-Heyting algebra obtained by adding to $L_1$ a new
element on the top. Clearly $L'$ belongs to $\cV_2$ by
fact~\ref{fa:prod-V2} and the embedding $\varphi$ extends uniquely to an
$\ltc$\--embedding of $L$ into $L'$, so we are done. 

{\it Case~3:} $a=\UN$ and $x=\ZERO$. Then $L_0=\Ldeux$ embeds into
$\Ldeux\times\Ldeux\in\cV_2$ which gives the conclusion. 

{\it Case~4:} $a=\UN$ and $x\neq\ZERO$. Then by assumption $b_1 \meet b _2 =
\ZERO$. Lemma~\ref{le:S1} applied to $x,b_1,b_2$ gives an
$\ltc$\--embedding $\varphi$ of $L_0$ into a finite co-Heyting algebra $L_1$
containing non-zero elements $x_1,x_2$ such that $x-x_1=x_2\geq b_2$,
$x-x_2=x_1\geq b_1$ and $x_1\meet x_2=\ZERO$. Just as in case~2,
the co-Heyting algebra $L^\dag$ obtained by adding to $L_1$ a new element
on the top belongs to $\cV_2$, and $\varphi$ extends to an embedding of $L$
into $L^\dag$. Now $\{a,\{x_1,x_2\},2\}$ is a signature in $L^\dag$ since
$x_1\join x_2=x$ is the predecessor of $a=\UN$ in $L^\dag$. Let $L'$ be
an extension generated over $L^\dag$ by a primitive tuple $(a_1,a_2)$ with
signature $(a,\{x_1,x_2\},2)$ (see remark~\ref{re:exist-sign}). By
construction $a_1\geq a_1\meet x=x_1\geq b_1$ and symmetrically $a_2\geq b_2$. We
also have $a=a_1\join a_2$ and $a_1\meet a_2=x_1\meet x_2=\ZERO$. By
theorem~\ref{th:ext-primitive} $a_1,a_2$ are exactly the two join
irreducible components of $\UN$ in $L'$ hence $L'$ belongs to $\cV_2$
by fact~\ref{fa:prod-V2}, $a-a_1=a_2$ and $a-a_2=a_1$, so we are done. 
\end{proof}

\begin{theorem}\label{th:EC-V2}
  Every co-Heyting algebra existentially closed in $\cV_2$
  satisfies the density axiom D2 and the splitting axiom S2.
\end{theorem}

\begin{proof}
By fact~\ref{fa:prod-V2} and lemma~\ref{le:non-sense-prod} this
follows directly from lemmas~\ref{le:D2} and \ref{le:S2}.
\end{proof}

Here is a partial converse of theorem~\ref{th:EC-V2}.

\begin{theorem}\label{th:embed-V2}
  Let $L$ be an algebra in $\cV_2$ satisfying the density axiom D2 and
  the splitting axiom S2. Let $L_0$ be a finite subalgebra of $L$ and
  $L_1$ be a finite algebra in $\cV_2$ containing $L_0$. Then there
  exists an embedding of $L_1$ into $L$ which fixes every point of
  $L_0$. 
\end{theorem}

\begin{proof}
By corollary~\ref{co:mini-prim} we can assume that $L_1$ is generated
over $L_0$ by a primitive tuple. Let $\sigma=(g,\{h_1,h_2\},r)$ be the
signature of $L_1$ in $L_0$. By corollary~\ref{co:isom-sign} it is
sufficient to prove that $\sigma$ is the signature of a primitive tuple of
elements $x_1, x_2 \in L$. 

{\it Case 1:} $r=1$ so $h_1=h_2$. Since $h_1 \leq g^- \ll g$ we have
$\UN-g^-=\UN$ hence obviously $h_1 \meet g^- \meet (\UN-(\UN-g^-))=\ZERO$. 
The splitting property $S2$ then applies to the elements $g,g^-,h_1$
in $L$. Then continue like in case~1 of the proof of
theorem~\ref{th:embed-V1}.

{\it Case 2:} $r=2$ so $h_1 \join h_2 = g^-$. If $\UN - g < \UN$ then 
$g$ is one of the join irreducible components of $\UN$ in $L_0$. By
theorem~\ref{th:ext-primitive} $x_1,x_2$ are then distinct join
irreducible components of $\UN$ in $L_1$, and since $L_1$ belongs to
$\cV_2$ it follows that $x_1 \meet x_2 = \ZERO$ and {\it a fortiori} $h_1 \meet
h_2 = \ZERO$. On the other hand if $\UN - g = \UN$ then obviously
$\UN-(\UN-g)=\ZERO$. So in any case we have: 
\begin{equation}\label{eq:sign-V2}
  h_1 \meet h_2 \meet (\UN - (\UN-g))=\ZERO 
\end{equation}
The splitting property S2 then applies in $L$ to the elements
$g,h_1,h_2$. Then continue like in case~2 of the proof of
theorem~\ref{th:embed-V1}.
\end{proof}

\begin{remark}\label{re:exist-sign-V2}
  The proof shows that the minimal extension of a finite co-Heyting
  algebra $L_0$ in $\cV_2$ determined by a signature $(g,\{h_1,h_2\},r)$
  belongs to $\cV_2$ if and only if either $r=1$, or $r=2$ and
  condition (\ref{eq:sign-V2}) holds. Also the analogues of
  corollaries~\ref{co:embed-V1-fini} and \ref{co:embed-V1-sat} hold
  for $\cV_2$ as a consequence of theorem~\ref{th:embed-V2}
\end{remark}

\section{Density and splitting in $\cV_3$}

We introduce the following axioms:
\begin{description}
  \item{\bf [Density D3]} 
    For every $a$ such that $a =\UN-(\UN-a) \neq \ZERO$ there exists a non zero
    element $b$ such that $b \ll a$.
  \item{\bf [Splitting S3]} 
    Same as S1.
\end{description}

A co-Heyting algebra $L$ belongs to $\cV_3$ if and only if it
has dimension $\leq 1$. If $L$ is finite this is equivalent to say that
every join irreducible element of $L$ is either maximal or minimal (or
both) in $\jirr(L)$. 

\begin{lemma}\label{le:D3}
  Let $a$ be any element of a finite algebra $L$ in $\cV_3$. If $a =
  \UN-(\UN-a) \neq \ZERO$ then there exists a finite algebra $L'$ in
  $\cV_3$ containing $L$ and a non zero element $b$ in $L'$ such that
  $b \ll a$.
\end{lemma}

\begin{proof}
Let $a_1,\dots,a_r$ be the join irreducible components of $a$ in $L$. 
The assumption that $a=\UN-(\UN-a) \neq \ZERO$ means that $r \neq 0$ and all
the $a_i$'s are join irreducible components of $\UN$, that is maximal
elements in $\jirr(L_0)$. If there exists $i \leq r$ such that $a_i$ is
not in the same time minimal in $L$ (that is $a_i$ is not an atom of
$L$) then we can choose $b \in \jirr(L)$ such that $b<a_i$. Then $b$
is non zero and $b \ll a_i$ because $a_i$ is join irreducible, so {\it a
fortiori} $b \ll a$. The conclusion follows, with $L'=L$.

It only remains to deal with the case when all the $a_i$ are both
maximal and minimal in $\jirr(L)$. But in this case the construction
of lemma~\ref{le:D1} (with $c=\ZERO$) gives an extension $L'$ on $L$
such that: 
\begin{itemize}
  \item
   $\jirr(L')=\jirr(L)\cup\{x_1,\dots,x_r\}$.
 \item 
   For every $i \leq r$ and every $x \in \jirr(L)$, $x \not< x_i$ and:
   \begin{displaymath}
     x_i < x \iff x=a_i 
   \end{displaymath}
\end{itemize}
So there are still no chain in $\jirr(L')$ containing more than two
distinct join irreducible elements, that is $L'$ belongs to $\cV_2$,
and clearly: 
\begin{displaymath}
   \ZERO\neq x_1\join\cdots\join x_r \ll a 
\end{displaymath}
\end{proof}

\begin{lemma}\label{le:S3}
  Let $a,b_1,b_2$ be elements of a finite algebra $L$ in $\cV_3$.
  If $b_1 \join b_2 \ll a \neq \ZERO$ then there exists a finite 
  algebra $L'$ in $\cV_3$ containing $L$ and non zero elements $a_1$, $a_2$ in
  $L'$ such that:
  \begin{displaymath}
    \begin{array}{c}  
      a-a_2 = a_1 \geq b_1\\
      a-a_1 = a_2 \geq b_2\\
      a_1 \meet a_2 = b_1 \meet b_2 
    \end{array}
  \end{displaymath}
\end{lemma}

\begin{proof}
Same proof as for lemma~\ref{le:S1}. Indeed, in the extension $L'$ of
$L$ constructed in that proof the maximal length of the chains of join
irreducible elements is the same as in $L$. So if $L$ belongs to
$\cV_3$ then so does $L'$.
\end{proof}

\begin{theorem}\label{th:EC-V3}
  The theory of the variety $\cV_3$ has a model-completion which is
  axiomatized by the axioms of co-Heyting algebras augmented by the
  density and splitting axioms D3 and S3. 
\end{theorem}

\begin{proof}
As for theorem~\ref{th:EC-V1} it immediately follows from
lemmas~\ref{le:D3} and \ref{le:S3}, {\it via}
lemma~\ref{le:non-sense-fin}, that every algebra existentially closed in
$\cV_3$ satisfies the axioms D3 and S3. 

For the converse, by fact~\ref{fa:mod-comp} it is sufficient to show
that given an algebra $L$ in $\cV_3$ satisfying D3 and S3, a finitely
generated subalgebra $L_0$ and a finitely generated extension $L_1$ of
$L_0$ in $\cV_3$ there exists an embedding of $L_1$ in $L$ which fixes
$L_0$ pointwise. Since $\cV_3$ is locally finite, $L_0$ and $L_1$ are
finite and by corollary~\ref{co:mini-prim} we can assume that
$L_1$ is generated by a primitive tuple $(x_1,x_2)$. Let
$\sigma=(g,\{h_1,h_2\},r)$ be the signature of $L_1$ in $L$, numbered so that
$h_i=x_i \meet g^-$. By corollary~\ref{co:isom-sign} we have to find a
primitive tuple in $L$ having signature $\sigma$. 

{\it Case 1:} $r=1$ so $x_1=x_2$ and $h_1 \ll x_1 \ll g$. Since $x_1,g$
are join irreducible in $L_1$ and since $L_1$ belongs to $\cV_3$,
necessarily $g$ is a join irreducible component of $\UN$, $x_1$ is an
atom of $L_1$, and consequently $h_1=\ZERO$. The splitting axiom
S3 applied to $g,g^-,\ZERO$ gives non zero elements $y_1,y_2$ in $L$
such that: 
\begin{displaymath}
  \begin{array}{c}
    g-y_1=y_2 \geq g^- \\
    g-y_2=y_1  \\
    y_1 \meet y_2 = \ZERO
  \end{array}
\end{displaymath} 
By construction $(y_1,y_2)$ is a primitive tuple over $L_0$ hence by
remark~\ref{re:g-prim} and theorem~\ref{th:ext-primitive} we have:
\begin{displaymath}
  \jirr(L_0\gen{y_1})=\big(\jirr(L_0)\setminus\{g\}\big)\cup\{y_1,y_2\}
\end{displaymath}
Since $g$ was a join irreducible component of $\UN$ in $L_0$, the same
then holds for $y_1,y_2$ in $L_0\gen{y_1}$. It follows that
$\UN-(\UN-y_1)=y_1$ hence the density axiom D3 gives $x \in L \setminus \{\ZERO\}$
such that $x \ll y_1$. {\it A fortiori} $x \ll g$ and by construction $x \meet
g^- \leq y_1 \meet y_2 = \ZERO$. It easily follows that $(x,x)$ is a
primitive tuple with signature $(g,\{\ZERO\},1)=\sigma$ in $L_0$. 

{\it Case 2:} $r=2$ so $h_1 \join h_2 = g^-$. The same construction as in
the case 2 of the proof of theorem~\ref{th:embed-V1} applies here and
gives the conclusion.
\end{proof}

\section{Density and splitting in $\cV_4$}

We introduce the following axioms:
\begin{description}
  \item{\bf [Density D4]} 
    Same as D3.
  \item{\bf [Splitting S4]} 
    Same as S1 with the additional assumption that 
    $b_1 \meet b_2 \meet (\UN-a) =\ZERO$.
\end{description}

\begin{fact}\label{fa:V4-SP}
  For any finite co-Heyting algebra $L$ the following conditions are
  equivalent. 
  \begin{enumerate}
    \item\label{it:V4-belong}%
      $L$ belongs to $\cV_4$.
    \item\label{it:V4-3-comp}%
      $L$ belongs to $\cV_3$
      (every element of $\jirr(L)$ is either maximal or minimal) 
      and for any three distinct join irreducible components
      $x_1,x_2,x_3$ of $\UN$, we have $x_1\meet x_2 \meet x_3=\ZERO$. 
    \item\label{it:V4-product}%
      $L$ $\ltc$\--embeds in a product of
      finitely many copies of $\Lcinq$.
  \end{enumerate}
\end{fact}

This is probably well known. For lack of a reference we give here
an elementary (and sketchy) proof. We can assume that $L\neq\Lun$
otherwise everything is trivial.

\begin{proof} 
(\ref{it:V4-product})$\Rightarrow$(\ref{it:V4-belong}) is clear.

(\ref{it:V4-belong})$\Rightarrow$(\ref{it:V4-3-comp})
Since $L$ belongs to $\cV_4$, which is generated by $\Lcinq$, which
belongs to $\cV_3$, obviously $L$ belongs to $\cV_3$. Now assume that
$\UN$ has at least three distinct join irreducible components
$x_1,x_2,x_3$ in $L$. The equation defining $\cV_4$ gives:
\begin{equation}\label{eq:V4-axiom}%
  (x_1-x_2)\meet(x_2-x_1)\meet(x_2 \bigtriangleup (\UN-x_1))=\ZERO
\end{equation}
We have $x_1-x_2=x_1$, $x_2-x_1=x_2$ and $\UN-x_1$ is the join of all
join irreducible components of $\UN$ except $x_1$. In particular it is
greater than $x_2$ and $x_3$ so we get:
\begin{displaymath}
  x_2 \bigtriangleup (\UN-x_1) = (\UN-x_1)-x_2 \geq x_3
\end{displaymath}
Finally (\ref{eq:V4-axiom}) becomes $x_1 \meet x_2 \meet
((\UN-x_1)-x_2)=\ZERO$ hence {\it a fortiori} $x_1 \meet x_2 \meet x_3 =
\ZERO$.

(\ref{it:V4-3-comp})$\Rightarrow$(\ref{it:V4-product}) 
We consider:
\begin{displaymath}
  \cI=\{(x_1,x_2)\in\jirr(L)\times\jirr(L)\tq 
            x_1<x_2\mbox{ or }x_1=x_2\mbox{ is an atom}\}
\end{displaymath}
$\cI$ is ordered as follows:
\begin{displaymath}
  (y_1,y_2)<(x_1,x_2) \iff y_1=y_2=x_1<x_2 
\end{displaymath}
The ordered set $\cI$ looks like $\jirr(L)$ except that every point of
$\jirr(L)$ strictly greater than $r$ atoms has been ``split'' in $r$
points strictly greater than only one atom. We ``collapse'' these $r$
points {\it via} the map $\pi$ defined for any $\xi=(x_1,x_2)\in\cI$ by
$\pi(\xi)=x_2$. This defines an $\ltc$\--embedding of $L$ into
$L'=\dec(\cI)$ by means of proposition~\ref{pr:finite-injection}. Then 
(\ref{it:V4-3-comp}) implies that $\cI$ is a finite disjoint union of
copies of sets represented in figure~\ref{fi:comp-V4}.
\begin{figure}[ht]
  \begin{center}
    \newcommand{\noir}[1]{\filldraw #1 circle (.1)}
    \begin{tikzpicture}[scale=.5]
      % Le V
      \draw (-6.5,1) -- (-5.5,0) -- (-4.5,1);
      \noir{(-6.5,1)};
      \noir{(-5.5,0)};
      \noir{(-4.5,1)};

      % Le I
      \draw (0,0) -- (0,1);
      \noir{(0,0)};
      \noir{(0,1)};

      % Le point.
      \noir{(5,.5)};
    \end{tikzpicture}
    \caption{\label{fi:comp-V4}The connected components of $\cI$}
  \end{center}
\end{figure}
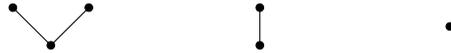
The family of all decreasing subsets of these sets are respectively
isomorphic to $\Lcinq$, $\Ltrois$ and $\Ldeux$. Since $\jirr(L')$ is
order-isomorphic to $\cI$, it follows that $L'$ is a direct product of
finitely many copies of these three algebras. Each of these copies
obviously $\ltc$\--embeds into $\Lcinq$ so we are done. 
\end{proof}

\begin{lemma}\label{le:D4}
  Let $a$ be any element of $\Lcinq$ such that $a = \UN-(\UN-a)$. 
  Then there exists an element $b$ in $\Lcinq$ such that
  $b \ll a$. If moreover $a\neq\ZERO$ then $b$ can be chosen non zero.
\end{lemma}

\begin{proof}
The assumption that $\UN-(\UN-a)=a\neq\ZERO$ implies that $a$ is not the
unique atom $c$ of $\Lcinq$. If $a=\ZERO$ one can take $b=\ZERO$.
Otherwise one can take $b=c$. 
\end{proof}

\begin{lemma}\label{le:S4}
  Let $a,b_1,b_2$ be any elements of $\Lcinq$ such that
  $b_1 \join b_2 \ll a$ and $b_1 \meet b_2 \meet (\UN-a) =\ZERO$. Then there
  exists an extension $L'$ of $L$ in $\cV_4$ and 
  elements $a_1$, $a_2$ in $L'$ such that:
  \begin{displaymath}
    \begin{array}{c}  
      a-a_2 = a_1 \geq b_1\\
      a-a_1 = a_2 \geq b_2\\
      a_1 \meet a_2 = b_1 \meet b_2 
    \end{array}
  \end{displaymath}
  If moreover $a\neq\ZERO$ then $a_1,a_2$ can be chosen both non zero.
\end{lemma}

\begin{proof}
Let $c$ denote the unique atom of $\Lcinq$. The first assumption on
$a,b_1,b_2$ implies that $b_1 \join b_2$ is either $\ZERO$ or $c$. In
particular we can always assume that $b_2 \leq b_1$. 

{\it Case 1:} $a=\ZERO$. One can take $a_1=a_2=\ZERO$.

{\it Case 2:} $b_1=b_2 =c$. By assumption $c \ll a$ and $c \meet (\UN
-a)=\ZERO$ hence $a=\UN$. So we can take $L'=\Lcinq$ and for $a_1,a_2$
the join irreducible components of $\UN$.

\begin{figure}[ht]
  \begin{center}
    \begin{tikzpicture}[scale=.5]
      \begin{scope}[every node/.style={draw,shape=circle,fill=white,
                     inner sep=0mm, minimum size=1mm}]
        \draw (0,0) node(O){} -- (-1,1) node(A1)[fill=black]{} 
          -- (0,2) node(A){} -- (-1,3) node{} -- (0,4) node{} 
          -- (1,3) node{} -- (0,2) 
          -- (1,1) node(A2)[fill=black]{} 
          -- (0,0);
      \end{scope}
      \draw (O) node[below]{$b_1=b_2=\ZERO$}
        (A1) node[left]{$a_1$}
        (A2) node[right]{$a_2$}
        (A) node[xshift=2mm,right]{$a$};

    \end{tikzpicture}
    \caption{\label{fi:V4-huit}Case 3}
  \end{center}
\end{figure}

{\it Case 3:} $a=c$. Then $b_1=b_2=\ZERO$ and we can take for $a_1$,
$a_2$ the atoms of the extension $L'$ of $\Lcinq$ shown in
figure~\ref{fi:V4-huit} (the white points are the points of $\Lcinq$).
Note that the four join irreducible elements of $L'=\Lcinq\times\Ldeux$
belongs to $\cV_4$.

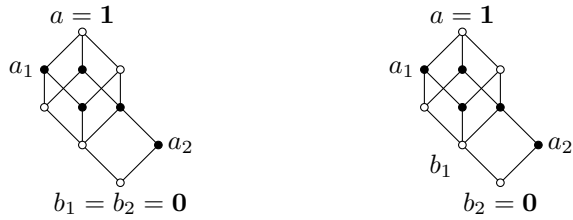
\begin{figure}[ht]
  \begin{center}

    \begin{tikzpicture}[scale=.5]

      % Dessin de base (gauche)
      \dessin

      % L\'egendes
      \draw (0,0) node[below]{$b_1=b_2=\ZERO$}
        (-1,4) node[above]{$a=\UN$}
        (-2,3) node[left]{$a_1$}
        (1,1) node[right]{$a_2$};

      \begin{scope}[xshift=10cm]

        % Dessin de base (droit)
        \dessin

        % L\'egendes
        \draw (0,0) node[below]{$b_2=\ZERO$}
          (-1,1) node[below left]{$b_1$}
          (-1,4) node[above]{$a=\UN$}
          (-2,3) node[left]{$a_1$}
          (1,1) node[right]{$a_2$};

        \end{scope}

    \end{tikzpicture}
  \end{center}
 \caption{\label{fi:V4-cas-45}Cases 4 and 5}
\end{figure}

\begin{figure}[ht]
  \begin{center}

    \begin{tikzpicture}[scale=.5]

      % Dessin de base (gauche)
      \dessin

      % L\'egendes
      \draw (0,0) node[below]{$b_1=b_2=\ZERO$}
        (0,3) node[right]{$a$}
        (1,1) node[right]{$a_2$};

      % Fl\`eche courbe
      \draw[->,densely dotted] (-3,.7) node[left]{$a_1$} to[bend right] (-1.1,1.9);

      \begin{scope}[xshift=10cm]

        % Dessin de base (droit)
        \dessin

        % Fl\`eche courbe
        \draw[->,densely dotted] (-3,.7) node[left]{$a_1$} to[bend right] (-1.1,1.9);

        % L\'egendes
        \draw (0,0) node[below]{$b_2=\ZERO$}
          (-1,1) node[below left,xshift=1mm]{$b_1$}
          (0,3) node[right]{$a$}
          (1,1) node[right]{$a_2$};

      \end{scope}

    \end{tikzpicture}
    \caption{\label{fi:V4-cas-67}Cases 6 and 7}
  \end{center}
\end{figure}
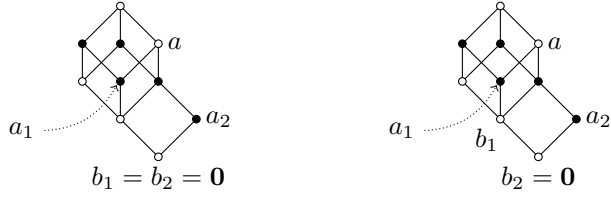

{\it Cases 4 to 7:} The four remaining cases when $a>c$ are summarised
in figures~\ref{fi:V4-cas-45} and \ref{fi:V4-cas-67}. In each case the
white points represent the points of $\Lcinq$ and one can take for
$a_1$, $a_2$ the points in the extension $L'$ of $\Lcinq$ shown in the
figures. Note that $L'$ is just $\Lcinq\times\Ltrois$ so it belongs to
$\cV_4$.
\end{proof}

\begin{theorem}\label{th:EC-V4}
  The theory of the variety $\cV_4$ has a model-completion which is
  axiomatized by the axioms of co-Heyting algebras augmented by the
  density and splitting axioms D4 and S4. 
\end{theorem}

\begin{proof}
As for theorem~\ref{th:EC-V3}, the only thing which it remains to
prove after lemmas~\ref{le:D4} and \ref{le:S4} is that: given an
algebra $L$ in $\cV_4$ satisfying D4 and S4, a finitely generated
subalgebra $L_0$ and a finitely generated extension $L_1$ of $L_0$ in
$\cV_4$ generated by a primitive tuple $(x_1,x_2)$ with signature
$\sigma=(g,\{h_1,h_2\},r)$ in $L_0$ (numbered so that $h_i=x_i \meet g^-$), there
exists a primitive tuple in $L$ having the same signature $\sigma$. 

{\it Case 1:} $r=1$. The same argument as in the case~1 in the proof
of theorem~\ref{th:EC-V3} applies here (when applying S4 in place of
S3 to $g,g^-,\ZERO$ the additional condition $g^-\meet\ZERO\meet(\UN-g)=\ZERO$ is
obviously satisfied).

{\it Case 2:} $r=2$ so $h_1 \join h_2 = g^-$. In order to apply the
splitting axiom S4 to $g,h_1,h_2$ we have to check that $h_1\meet h_2\meet
(\UN-g)=\ZERO$. Assume the contrary. Then $h_1$, $h_2$ are 
non zero so $g$ is not an atom. Since $L_0$ belongs to $\cV_4\subseteq\cV_3$ it
follows that $g$ is maximal in $\jirr(L_0)$ hence so are $x_1$, $x_2$
in $\jirr(L_1)$ (see theorem~\ref{th:ext-primitive}). With other
words $x_1$, $x_2$ are two distinct join irreducible components of
$\jirr(L_1)$ and $\UN-g$ is the join of all the other join irreducible
components of $\UN$ in $L_1$. But for any such component $x_3$ we must
have $x_1 \meet x_2 \meet x_3=\ZERO$ by fact~\ref{fa:V4-SP} so $x_1\meet x_2 \meet
(\UN-g)=\ZERO$. Since each $h_i \leq x_i$ this contradicts our
assumption. 

So we can apply S4 to $g,h_1,h_2$ and it gives $y_1,y_2$ in $L$. Then
finish like in the case~2 of the proof of theorem~\ref{th:embed-V1}.
\end{proof}

\section{Density and splitting in $\cV_5$}

The density and splitting axioms for $\cV_5$ are respectively D5$=$D3
and S5$=$S2. 

\begin{fact}\label{fa:V5-caract}  
  For a finite co-Heyting algebra $L$ the following conditions are
  equivalent:
  \begin{enumerate}
    \item
      $L$ belongs to $\cV_5$.
    \item 
      $L$ belongs to $\cV_2$ and $\cV_3$, that is every join
      irreducible element of $L$ which is not an atom is a join
      irreducible component of $\UN$, and for any two distinct join
      irreducible components $x_1$, $x_2$ of $\UN$ we have $x_1 \meet
      x_2=\ZERO$.
    \item 
      $L$ $\ltc$\--embeds in a direct product of finitely many copies
      of the three elements co-Heyting algebra $\Ltrois$. 
  \end{enumerate}
\end{fact}

This is probably well known, and anyway the adaptation to this context
of the proof that we gave for the analogous fact~\ref{fa:V4-SP} is
straightforward. 

\begin{theorem}\label{th:EC-V5}
  The theory of the variety $\cV_5$ has a model-completion which is
  axiomatized by the axioms of co-Heyting algebras augmented by the
  density and splitting axioms D5 and S5. 
\end{theorem}

\begin{proof}
Let $c$ denote the unique atom of $\Ltrois$. 

The only elements $a$ in $\Ltrois$ such that $a =\UN-(\UN-a)$ are
$\ZERO$ and $\UN$. Clearly if $a=\ZERO$ then $b=\ZERO$ satisfies $b \ll
a$, and otherwise $b=c$ satisfies $\ZERO\neq b \ll a$. 
By fact~\ref{fa:V5-caract} and lemma~\ref{le:non-sense-prod} it
follows that every algebra existentially closed in $\cV_5$ satisfies 
D5.

Now let $a,b_1,b_2$ in $L_3$ be such that $b_1 \join b_2 \ll a$ and $b_1 \meet
b_2 \meet (\UN-(\UN-a))=\ZERO$. If $a=\ZERO$ then one can take
$a_1=a_2=\ZERO$ as a solution for the conclusion of S5. Otherwise
three cases may happen: 

{\it Case 1:} $a=c$ and $b_1=b_2=\ZERO$.
             
{\it Case 2:} $a=\UN$ and $b_1=b_2=\ZERO$. 
             
{\it Case 3:} $a=\UN$ and (renaming $b_1$ and $b_2$ if necessary)
$b_1=c$ and $b_2=\ZERO$. 

In each of these cases one can take for $a_1$, $a_2$
the elements of the extension $\Lneuf$ of $\Ltrois$ shown in
figure~\ref{fi:V5-embed-L3} (the white points represent $\Ltrois$).
Note that $\Lneuf=\Ltrois\times\Ltrois$ belongs to $\cV_5$. By
fact~\ref{fa:V5-caract} and lemma~\ref{le:non-sense-prod} again, it
follows that every algebra existentially closed in $\cV_5$ satisfies
S5.

\begin{figure}[ht]
  \begin{center}
    
    \begin{tikzpicture}[scale=.5]
      
      % Figure de gauche
      \begin{scope}[every node/.style={draw,shape=circle,fill=black,
               inner sep=0mm, minimum size=1mm}]
        \foreach \k in {0,1,2}
          { \draw (\k,\k-2) -- ++(-2,2);
            \draw (-\k,\k-2) node{} -- ++(1,1) node{} -- ++(1,1) node{}; }
      \end{scope}
      \begin{scope}[every node/.style={draw,shape=circle,fill=white,
               inner sep=0mm, minimum size=1mm}]
        \draw (0,-2) node(O){} (0,0) node(a){} (0,2) node(UN){};
      \end{scope}
      \draw (-1,-1) node[below left]{$a_1$}
        (1,-1) node[below right]{$a_2$}
        (a) node[right]{$a$}
        (O) node[below]{$b_1=b_2=\ZERO$};

      % Figure du milieu
      \begin{scope}[xshift=7cm]
        \begin{scope}[every node/.style={draw,shape=circle,fill=black,
                 inner sep=0mm, minimum size=1mm}]
          \foreach \k in {0,1,2}
            { \draw (\k,\k-2) -- ++(-2,2);
              \draw (-\k,\k-2) node{} -- ++(1,1) node{} -- ++(1,1) node{}; }
        \end{scope}
        \begin{scope}[every node/.style={draw,shape=circle,fill=white,
                 inner sep=0mm, minimum size=1mm}]
                 \draw (0,-2) node(O){} (0,0) node{} (0,2) node(UN){};
        \end{scope}
        \draw (-2,0) node[left]{$a_1$}
          (2,0) node[right]{$a_2$}
          (UN) node[above]{$a=\UN$}
          (O) node[below]{$b_1=b_2=\ZERO$};
      \end{scope}

     % Figure de droite
      \begin{scope}[xshift=14cm]
        \begin{scope}[every node/.style={draw,shape=circle,fill=black,
                 inner sep=0mm, minimum size=1mm}]
          \foreach \k in {0,1,2}
            { \draw (\k,\k-2) -- ++(-2,2);
              \draw (-\k,\k-2) node{} -- ++(1,1) node{} -- ++(1,1) node{}; }
        \end{scope}
        \begin{scope}[every node/.style={draw,shape=circle,fill=white,
                 inner sep=0mm, minimum size=1mm}]
          \draw (0,-2) node(O){} (-1,-1) node(b){} (0,2) node(UN){};
        \end{scope}
        \draw (UN) node[above]{$a=\UN$}
          (-1,-1) node[below left]{$b_1$}
          (-2,0) node[left]{$a_1$}
          (2,0) node[right]{$a_2$}
          (O) node[below]{$b_2=\ZERO$};
      \end{scope}

    \end{tikzpicture}

    \caption{\label{fi:V5-embed-L3}$\Ltrois\subset\Lneuf$ gives solutions for S5 (three possible cases).}
  \end{center}
\end{figure}

Conversely let $L$ in $\cV_5$ satisfying D5 and S5, $L_0$ a finitely
generated subalgebra and $L_1$ a finitely generated extension of $L_0$
in $\cV_5$ generated by a primitive tuple $(x_1,x_2)$ with signature
$\sigma=(g,\{h_1,h_2\},r)$ in $L_0$ (numbered so that $h_i=x_i \meet g^-$). As
usually it only remains to find a primitive tuple in $L$ having the
same signature $\sigma$ in order to conclude that $L_1$ embeds into $L$
over $L_0$ by corollary~\ref{co:isom-sign}, hence to finish the proof
by fact~\ref{fa:mod-comp}. 

{\it Case 1:} $r=1$ so $x_1=x_2$ and $h_1 \ll x_1 \ll g$. Same as 
case~1 in the proof of theorem~\ref{th:EC-V3}.

{\it Case 2:} $r=2$ so $h_1 \join h_2 = g^-$. Same as case~2 in the proof
of theorem~\ref{th:EC-V2} (note that $\cV_5$ is contained in $\cV_2$
when applying this proof).
\end{proof}

\section{Density and splitting in $\cV_6$}%
\label{se:V6}%

We introduce our last axioms.
\begin{description}
  \item{\bf [Density D6]} 
    Same as D1.
  \item{\bf [Splitting S6]}
    Same as S1 with the additional assumption that $b_1 \meet b_2=\ZERO$.
\end{description}

\begin{fact}\label{fa:V6-caract}
  A finite co-Heyting algebra belongs to $\cV_6$ if and only if it
  embeds into a direct product of finitely many finite chains.
\end{fact}

This is certainly well known, and easy to check.  

\begin{theorem}\label{th:EC-V6}
  The theory of the variety $\cV_6$ has a model-completion which is
  axiomatized by the axioms of co-Heyting algebras augmented by the
  density and splitting axioms D6 and S6. 
\end{theorem}

\begin{proof}
Let $a,c$ be any elements in a finite chain $L$ such that $c \ll a$. If
$a=\ZERO$ then $b=\ZERO$ satisfies $c \ll b \ll a$. Otherwise $c <a$ and
obviously $L$ embeds into a chain $L'$ containing a new intermediate
element $b$ between $a$ and $a^-$. Then by construction $c \ll b \ll a$
and $b\neq\ZERO$. By fact~\ref{fa:V6-caract} and
lemma~\ref{le:non-sense-prod} it follows that every algebra
existentially closed in $\cV_6$ satisfies D6. 

Let $a,b_1,b_2$ be three elements in a finite chain $L$ such that $b_1
\join b_2 \ll a$ and $b_1\meet b_2 =\ZERO$. We may assume that $b_2 \leq b_1$, so
by assumption $b_2=\ZERO$. If $a=\ZERO$ then $a_1=a_2=\ZERO$ satisfy
the conclusion of S6. Otherwise $b_1 < a$ and one can take for
$a_1,a_2$ the non zero points in the extension $L'$ of $L$ shown in
figure~\ref{fi:V6-embed} (the  white points represent
$L$). Note that $L'=L\times\Ldeux$ belongs to $\cV_2$.
By fact~\ref{fa:V6-caract} and lemma~\ref{le:non-sense-prod} again,
it follows that every algebra existentially closed in $\cV_6$
satisfies S6. 

\begin{figure}[ht]
  \begin{center}

    \begin{tikzpicture}[scale=.5]
      
      % Ar\^etes courtes
      \foreach \k in {0,1,3.5,4.5}  %,5.5,7.5
        {\draw (\k,\k) -- ++(-1,1);}

      % Ar\^ete longue du dessous + points 
      \begin{scope}[every node/.style={draw,shape=circle,fill=white,
                 inner sep=0mm, minimum size=1mm}]
        \draw (0,0) node{} -- (1,1) node{} -- (1.75,1.75);
        \draw[loosely dotted] (1.75,1.75) -- (2.75,2.75);
        \draw (2.75,2.75) -- (3.5,3.5) node{} -- (4.5,4.5);
      \end{scope} 
      \begin{scope}[every node/.style={draw,shape=circle,fill=black,
                 inner sep=0mm, minimum size=1mm}]
        \draw (4.5,4.5) node(a1){}; %-- (5.5,5.5) node{} -- (6,6);
%         \draw[loosely dotted] (6,6) --(7,7);
%         \draw (7,7) -- (7.5,7.5) node{};
      \end{scope}

      % Ar\^ete longue du dessus + points 
      \begin{scope}[shift={(-1,1)}]
        \begin{scope}[every node/.style={draw,shape=circle,fill=black,
                   inner sep=0mm, minimum size=1mm}]
          \draw (0,0) node(a2){} -- (1,1) node{} -- (1.75,1.75);
          \draw[loosely dotted] (1.75,1.75) -- (2.75,2.75);
          \draw (2.75,2.75) -- (3.5,3.5) node{} -- (4.5,4.5);
        \end{scope} 
        \begin{scope}[every node/.style={draw,shape=circle,fill=white,
                   inner sep=0mm, minimum size=1mm}]
          \draw (4.5,4.5) node(a){} -- (5.5,5.5) node{} -- (6,6);
          \draw[loosely dotted] (6,6) --(7,7);
          \draw (7,7) -- (7.5,7.5) node{};
        \end{scope}
      \end{scope}

      % \'Etiquettes
    \draw (a1) node[below right]{$a_1$};
    \draw (a2) node[left]{$a_2$};
    \draw (a) node[above left]{$a$};

    \end{tikzpicture}

    \caption{\label{fi:V6-embed}$L$ (in white) inside $L\times\Ldeux$}
  \end{center}
\end{figure}

Conversely let $L$ in $\cV_6$ be satisfying D6 and S6, $L_0$ a finitely
generated subalgebra and $L_1$ a finitely generated extension of $L_0$
in $\cV_6$ generated by a primitive tuple $(x_1,x_2)$ with signature
$\sigma=(g,\{h_1,h_2\},r)$ in $L_0$ (numbered so that $h_i=x_i \meet g^-$). As
usually it only remains to find a primitive tuple in $L$ having the
same signature $\sigma$ in order to conclude that $L_1$ embeds into $L$
over $L_0$ by corollary~\ref{co:isom-sign}, hence to finish the proof
by fact~\ref{fa:mod-comp}. 

{\it Case 1:} $r=1$ so $x_1=x_2$ and $h_1 \ll x_1 \ll g$. Same as 
case~1 in the proof of theorem~\ref{th:embed-V1}.

{\it Case 2:} $r=2$ so $h_1 \join h_2 = g^-$. Since $x_1,x_2$ are join
irreducible and incomparable, $x_1-x_2=x_1$ and $x_2-x_1=x_2$. By
definition of $\cV_6$ it follows that $x_1 \meet x_2=\ZERO$, hence {\it a
fortiori} $h_1 \meet h_2=\ZERO$. So the splitting axiom S6 applies to
$g,h_1,h_2$. Then finish the proof like in case~2 of
theorem~\ref{th:embed-V1}.
\end{proof}

\section{Appendix} 

It is proven in \cite{ghil-zawa-1997}, page~44, that for every $x,z$ in
an existentially closed algebra $L$ in $\cH_1$ there are elements
$x_1, x_2$ such that $x_1\join x_2= \UN$, $x_1\meet x_2=x$ and:
\begin{displaymath}
   (z-x_1)\meet x = (z-x_2)\meet x
\end{displaymath}
Since this axiom asserts the existence of a splitting of $\UN$ in two
parts $x_1$ and $x_2$ intersecting along $x$ with an additional
condition, it is very close in spirit to our axiom S1. Is it
equivalent to S1? Our guess is no. However it follows from the next
proposition (with $w=\UN$) that the above axiom is implied by S1. 

\begin{proposition}%\label{pr:}
  Let $L$ be a Heyting algebra satisfying the splitting axiom S1. Then
  for every $w,x,z$ in $L$ such that $x\leq w$ there are elements $x_1,x_2$ in $L$
  such that $x_1\join x_2= w$, $x_1\meet x_2=x$ and:
  \begin{displaymath}
     (z-x_1)\meet x = (z-x)\meet x = (z-x_2)\meet x
  \end{displaymath}
\end{proposition}

\begin{proof}
Since $(z-x)\meet x\ll z-x$, S1 gives\footnote{Of course S1 applies only if
$z-x\neq\ZERO$ but otherwise it suffices to take $a_1=a_2=\ZERO$.}
elements $a_1,a_2$ such that: 
\begin{displaymath}
  \begin{array}{c} 
    (z-x)-a_2 = a_1 \geq (z-x)\meet x\\
    (z-x)-a_1 = a_2 \geq (z-x)\meet x\\
    a_1 \meet a_2 = (z-x)\meet x
  \end{array}
\end{displaymath}
Let $c=w-(z\join x)$. Since $a_i \leq z-x\leq z\join x$ for $i=1,2$ we have that
$c\meet a_i\ll c$. Thus S1 again gives\footnote{As above, if $c=\ZERO$ we
cannot apply S1 but $c_1=c_2=\ZERO$ then suits perfectly our needs.}
elements $c_1,c_2$ such that: 
\begin{displaymath}
  \begin{array}{c} 
    c-c_2 = c_1 \geq c\meet a_1\\
    c-c_1 = c_2 \geq c\meet a_2\\
    c_1 \meet c_2 = c\meet a_1 \meet a_2
  \end{array}
\end{displaymath}
Let $x_i=x\join a_i \join c_i$ for $i=1,2$. By construction:
\begin{displaymath}
  x_1 \join x_2 = x\join (a_1\join a_2)\join (c_1\join c_2)
  = x \join (z-x) \join c = (z \join x) \join c =w 
\end{displaymath}
Moreover $c_1 \meet a_2 = c_1 \meet c \meet a_2 \leq c_1 \meet c_2$. The latter is smaller
than $a_1 \meet a_2$ which is smaller than $x$. Symmetrically $c_2 \meet a_2 \leq
x$ so by distributivity we get:
\begin{displaymath}
  (a_1 \join c_1)\meet(a_2 \join c_2) \leq x
\end{displaymath}
Thus $x_1 \meet x_2 = x \join \big[(a_1 \join c_1)\meet(a_2 \join c_2)\big]=x$. Finally
we have by construction: 
\begin{displaymath}
  z-x_1=\big((z-x)-a_1\big)-c_1=a_2-c_1
\end{displaymath}
We already noticed that $c_1 \meet a_2 \leq c_1 \meet c_2$. The latter is smaller
than $a_1\meet a_2 \ll a_2$ so $a_2-c_1=a_2$. Recall that: 
\begin{displaymath}
  (z-x)\meet x \leq a_2 \leq z-x
\end{displaymath}
Thus $(z-x_1)\meet x=a_2\meet x = (z-x)\meet x$ and symmetrically for $x_2$.
\end{proof}

All the other properties of non-zero existentially closed Heyting
algebras listed in proposition~A2~(i)--(iv) of \cite{ghil-zawa-1997}
easily follow from the density axiom D1, except (iv) that we derive
from S1 in the next proposition. 

\begin{proposition}
  Let $x,y$ be any elements in a co\--Heyting algebra $L$ satisfying
  S1. Then $y\to x$ exists in $L$ if and only if $(\UN-y)\meet y\leq x$.
\end{proposition}

\begin{remark}
  It is an easy exercise to check that in every co\--Heyting algebra,
  if $(\UN-y)\meet y\leq x$ then $y\to x$ exists and equals $(\UN-y)\join x$.
  So the above proposition shows that among co\--Heyting algebras,
  those which satisfy the axiom S1 are ``the least possibly bi\--Heyting''.
\end{remark}

\begin{figure} [ht]
  \begin{center}

    \begin{tikzpicture}[scale=.5]

      % UN
      \filldraw[fill=gray!20] (-5,-3) rectangle (5,3);
      \draw (5,-3) node[above left]{$\UN$};

      % z
      \draw[thick,dotted,fill=white]
        (-.49,.7) .. controls +(1.5,2.5) and +(0,3) .. (-4,0)
        .. controls +(0,-4) and +(3,-2) .. (-.49,-.7) --cycle;
        \draw (-1.5,-2.3) node{$z$};

      % x
      \filldraw[fill=gray!30,opacity=.7] (-2.5,-1.5) rectangle (-.5,1.5);
      \draw (-2.5,-1.5) node[above right]{$x$};

      % y
      \filldraw[thick,fill=gray!80,opacity=.7] (-1.2,-.7) rectangle (2,.7);
      \draw (2,-.7) node[above left]{$y$};

      % a2, a1
      \filldraw[dotted,fill=white] 
        (3.5,.7) node{$a_2$} ellipse (.8 and 1.6);
      \draw (3,-1.5) node{$a_1$};

    \end{tikzpicture}
    \caption{\label{fi:appendice}Splitting of $\UN-(z \join y)$ (here
     $x \subseteq z$) }
  \end{center}
\end{figure}

As we explained at the beginning of this paper, many of our proofs are
inspired by the geometric intuition coming from the ``co\--Heyting''
(instead of ``Heyting''). As an illustration, we add the ``picture of
the proof'' and how to use it for this last proof.

By the above remark we only have to prove that, assuming $(\UN-y)\meet y\nleq
x$, the set $\cal Z$ of elements $z$ in $L$ such that $z\meet y\leq x$ has
(thanks to S1) no greatest element. So let $z$ be any element in $\cal
Z$, and let us imagine that $x,y,z$ are semi-algebraic subsets of the
real plane in figure~\ref{fi:appendice}.

By assumption $z \cap y \subseteq x$, that is $z$ does not contain any point of
$y$ which is not in $x$. The largest possible such set is the
complement of $y \setminus x$, but $z$ cannot be so large without meeting the
frontier of $y$, that is $(\UN-y)\cap y$, outside $x$. One sees then in
figure~\ref{fi:appendice} how to increase $z$ without changing $z \cap
y$: it suffices to split (using S1) the intermediate piece which is the
complement of $z \cup y$ into two disjoint pieces, one of which avoids to
touch the border, and to add the latter to $z$.

\begin{proof}
Let $x,y\in L$ such that $(\UN-y)\meet y\nleq
x$, and $\cal Z$ the set of elements $z$ in $L$ such that $z\meet y\leq x$.
We have to prove that $\cal Z$ has no greatest element. For any element
$z$ in $\cal Z$, let $a=\UN-(z \join y)$. Note that:
\begin{displaymath}
   [(z\join y)-y]\meet y = (z-y)\meet y\leq z\meet y\leq x
\end{displaymath}
Since $(\UN-y)\meet y\nleq x$ by assumption, it follows that $\UN\neq z\join y$
hence $a\neq\ZERO$. The splitting property S1 then gives non-zero
elements $a_1,a_2$ in $L$ such that:
\begin{displaymath}
  \begin{array}{c} 
    a-a_2 = a_1 \geq a \meet (z\join y)\\
    a-a_1 = a_2 \geq \ZERO\\
    a_1 \meet a_2 = \ZERO
  \end{array}
\end{displaymath}
Clearly $a_2\nleq z$ since $a-z=a$ and $a-a_2<a$. On the other hand: 
\begin{displaymath}
  a_2\meet y =a_2\meet a\meet y \leq a_2 \meet a_1 =\ZERO 
\end{displaymath}
Thus $(z \join a_2)\meet y=z\meet y\leq x$, which proves that $z \join a_2\in\cal Z$ and
consequently that $z$ is not maximal in $\cal Z$. 
\end{proof}

% \bibliographystyle{alpha}
% \bibliography{biblio}

\begin{thebibliography}{McK68}

\bibitem[CK90]{chan-keis-1990}
C.~C. Chang and H.~J. Keisler.
\newblock {\em Model theory}, volume~73 of {\em Studies in Logic and the
  Foundations of Mathematics}.
\newblock North-Holland Publishing Co., Amsterdam, third edition, 1990.

\bibitem[DJ11]{darn-junk-2011}
Luck Darni\`ere and Markus Junker.
\newblock Codimension and pseudometric in co-{H}eyting algebras.
\newblock {\em Algebra Universalis}, 64(3):251--282, 2011.

\bibitem[GZ97]{ghil-zawa-1997}
Silvio Ghilardi and Marek Zawadowski.
\newblock Model completions and r-{H}eyting categories.
\newblock {\em Ann. Pure Appl. Logic}, 88(1):27--46, 1997.

\bibitem[Hod97]{hodg-1997}
Wilfrid Hodges.
\newblock {\em A shorter model theory}.
\newblock Cambridge University Press, Cambridge, 1997.

\bibitem[Hos67]{hoso-1967}
Tsutomu Hosoi.
\newblock On intermediate logics. {I}.
\newblock {\em J. Fac. Sci. Univ. Tokyo Sect. I}, 14:293--312, 1967.

\bibitem[Mak77]{maks-1977}
L.~L. Maksimova.
\newblock Craig's theorem in superintuitionistic logics and amalgamable
  varieties.
\newblock {\em Algebra i Logika}, 16(6):643--681, 741, 1977.

\bibitem[McK68]{mcka-1968}
C.~G. McKay.
\newblock The decidability of certain intermediate propositional logics.
\newblock {\em The Journal of Symbolic Logic}, 33(2):258--264, 1968.

\bibitem[MT46]{mcki-tars-1946}
J.~C.~C. McKinsey and Alfred Tarski.
\newblock On closed elements in closure algebras.
\newblock {\em Ann. of Math. (2)}, 47:122--162, 1946.

\bibitem[Pit92]{pitt-1992}
Andrew~M. Pitts.
\newblock On an interpretation of second-order quantification in first-order
  intuitionistic propositional logic.
\newblock {\em J. Symbolic Logic}, 57(1):33--52, 1992.

\bibitem[Whe76]{whee-1976}
William~H. Wheeler.
\newblock Model-companions and definability in existentially complete
  structures.
\newblock {\em Israel J. Math.}, 25(3-4):305--330, 1976.

\end{thebibliography}

\end{document}